\documentclass[12pt]{article}   
\usepackage{graphicx,psfrag}
\usepackage{tikz}
\usepackage{tkz-graph}
\usepackage{color}
\usepackage{array}
\usepackage{colortbl}
\usepackage{amsmath,amsfonts}
\usepackage{amssymb,amsthm}

\newtheorem{theorem}{Theorem}
\newtheorem{corollary}[theorem]{Corollary}

\newtheorem{conjecture}[theorem]{Conjecture}
\newtheorem{example}[theorem]{Example}
\newtheorem{lemma}[theorem]{Lemma}
\newtheorem{problem}[theorem]{Problem}

\def\qed{\vbox{\hrule
 \hbox{\vrule\hbox to 5pt{\vbox to 8pt{\vfil}\hfil}\vrule}\hrule}}

\def\endproof{\unskip \nobreak \hskip0pt plus 1fill \qquad \qed \par}

\newcommand{\Row}{\mbox{\rm Row}}

\newcommand{\row}{\mbox{\rm row}}
\newcommand{\col}{\mbox{\rm col}}

\newcommand{\mb}{\mathbb}

\usepackage{tikz}
\usetikzlibrary{decorations.markings}

\begin{document}

\title{Alternating Sign Matrices and Hypermatrices, and a Generalization of Latin Squares}

\author{
 Richard A. Brualdi\footnote{Department of Mathematics, University of Wisconsin, Madison, WI 53706, USA. {\tt brualdi@math.wisc.edu}} \\
 Geir Dahl\footnote{Department of Mathematics,  
 University of Oslo, Norway.
 {\tt geird@math.uio.no}} }

\maketitle

\begin{abstract}
 An alternating sign matrix, or ASM,  is a $(0, \pm 1)$-matrix where the nonzero entries in each row and column alternate in sign. We generalize this notion to hypermatrices: an $n\times n\times n$ hypermatrix $A=[a_{ijk}]$ is an {\em alternating sign hypermatrix}, or ASHM, if each of its planes, obtained by fixing one of the three indices, is an ASM. Several results concerning ASHMs are shown, such as finding the maximum number of nonzeros of an $n\times n\times n$ ASHM, and  properties  related to Latin squares. Moreover, we investigate completion problems, in which one asks if a subhypermatrix can be completed (extended) into an ASHM. We show several theorems of this type.
 \end{abstract}


\noindent {\bf Key words.} Alternating sign matrix, hypermatrix, completion.

\noindent
{\bf AMS subject classifications.} 05B20, 15A69, 15B48.

\section{Introduction}

Let $A$ be an $n\times n$ $(0,\pm 1)$-matrix. Then $A$ is an {\it alternating sign matrix}, abbreviated ASM, provided in each of the $2n$ lines of $A$, that is, its rows and columns,
the nonzeros alternate beginning and ending with a $+1$. Permutation matrices are ASMs without any $-1$'s. 
ASMs were defined by Mills, Robbins, and  Ramsey \cite{MRR83} and have a fascinating history which can be found in \cite{Bressoud1999}.  Extending some of the work reported in \cite{Behrend07} and \cite{Striker09}, we carried out  in \cite{BrualdiDahl17}  a recent study of ASMs and related matrix classes and polyhedra  
where additional references, besides those given here, can be found.
Our goal is to generalize ASMs to three-dimensional matrices called {\it hypermatrices}. In doing so, we were led to a fascinating generalization of classical latin squares.

Let $A=[a_{ijk}]$ be an $n\times n\times n$ hypermatrix.
We refer to $i$ as the {\it row index}, $j$ as the {\it column index},
and $k$ as the {\it vertical index} of the hypermatrix $A$.
Then $A$ has three types of lines, each of cardinality $n$:
\begin{itemize}
\item[\rm (i)] The {\it row lines} (variable row index)
$ A_{*jk}=[a_{ijk}: i=1,2,\ldots n], \;(1\le j,k\le n)$;
\item[\rm (ii)] The {\it column lines} (variable column index)
$ A_{i*k}=[a_{ijk}: j=1,2,\ldots n], \;(1\le i,k \le n)$;
\item[\rm (iii)] The {\it vertical lines} (variable vertical index)
$ A_{ij*}=[a_{ijk}: k =1,2,\ldots n], \;(1\le i,j\le n)$.
\end{itemize}
Similarly, $A$ has three types of planes, each of cardinality $n^2$:
\begin{itemize}
\item[\rm (i)] The {\it horizontal-planes} (or {\it row-column-planes}) (variable row and column indices)
$ A_k^{\rm h}=A_{**k}=[a_{ijk}: i,j=1,2,\ldots n], \;(1\le k \le n)$;
\item[\rm (ii)] The {\it column-vertical-planes}  (variable row and vertical indices)
$ A_j^{\rm cv}=A_{*j*}=[a_{ijk}: i,k=1,2,\ldots n], \;(1\le j\le n)$;
\item[\rm (iii)] The {\it row-vertical-planes}  (variable  column and vertical indices)
$ A_i^{\rm rv}=A_{i**}=[a_{ijk}:j,k =1,2,\ldots n], \;(1\le i\le n)$.
\end{itemize}
The intersection of two planes of different types is a line; for instance, the intersection of a horizontal-plane with a row-vertical plane is a column line:
\[ A_{**k}\hspace{.05in} \cap\hspace{.05in} A_{i**}=
A_{i*k}.\]
We usually denote the $n\times n\times n$ hypermatrix $A$ by
\[A=[A_1^{\rm h},A_2^{\rm h},\ldots,A_n^{\rm h}],\mbox{ abbreviated to } A=[A_1,A_2,\ldots,A_n]\] where  the $A_i$ are the horizontal-planes $A_{**k}$.
To denote the fact that $A$ is a 3-dimensional array, we also write
\[A=A_1\nearrow A_2\nearrow\cdots\nearrow A_n\]
where the north-east arrow  $A_i\nearrow A_{i+1}$ is read as $A_i$ is below $A_{i+1}$ (or $A_{i+1}$ is on top of $A_i$). We can also write
\[A=[A_1^{\rm cv},A_2^{\rm cv},\ldots,A_n^{\rm cv}] \;
\mbox{ and } \;A=[A_1^{\rm rv},A_2^{\rm rv},\ldots,A_n^{\rm rv}].\] 

We can generalize an alternating sign matrix to an $n\times n\times n$  $(0,\pm 1)$-hypermatrix $A$ in two natural ways:
\begin{itemize}
\item[\rm (a)] $A$ is an {\it alternating sign hypermatrix}, abbreviated ASHM, provided in each of its  $3 n^2$ lines, the nonzeros alternate beginning and ending with a $+1$, equivalently, each of its $3n$ planes is an ASM.
\item[\rm (b)] $A$ is a  {\it  planar alternating sign hypermatrix}, abbreviated PASHM, provided each of its $n\times n$ horizontal planes is an ASM and each of its vertical lines sums to 1. (We can replace horizontal planes with any of the other two types of planes and vertical lines with the appropriate lines.)
\end{itemize}

The property of being a PASHM is weaker than being a ASHM.
If in our definition of a PASHM we replace the condition that the vertical lines sum to 1 with the condition that the nonzeros in the vertical lines alternate beginning and ending with a $+1$, then the result would be an ASHM. The weaker condition that the vertical lines sum to 1 connects the horizontal-plane  ASMs $A_1,A_2,\ldots,A_n$ with one another.

If $[A_1, A_2,\ldots,A_n]$ is a ASHM (respectively, a PASHM) then so is the hypermatrix $[A_n, \ldots,A_2,A_1]$ obtained by reversing the order of its horizontal planes. It follows that for $i\le n/2$, the ASMs $A_i$ and $A_{n-i}$ satisfy similar properties. More generally, the set of $n\times n\times n$ ASHMs is invariant under the symmetry group of the unit cube. This group, called the {\it hyperoctahedral group}\footnote{It is also the group of signed permutations $\pi$ of $\{\pm 1,\pm 2,\pm 3,\pm 4\}$ satisfying $\pi(-i)=-\pi(i)$ for all $i$.} is the direct product ${\mathcal S}_2\times {\mathcal S}_4$ of the symmetric permutation groups ${\mathcal S}_2$ and ${\mathcal S}_4$ and  has size 48.

An $n\times n\times n$ {\it permutation hypermatrix}  is a 
$(0,1)$-hypermatrix with exactly one 1 in each line, and these are the ASHMs 
not having any $-1$'s. 
If $A=[A_1,A_2,\ldots,A_n]$ is a PASHM, then  $A_1$ and $A_n$ and, in fact, each of the first  and last planes of each of the three types, must be permutation matrices. Thus, thinking of $A$ as an $n\times n\times n$ cube, the six boundary facets of the  cube  are $n\times n$ permutation matrices. 

We remark that the question of generalizing the Birkhoff - von Neumann theorem (which asserts that the extreme points of the polytope of $n\times n$ doubly stochastic matrices are the $n\times n$ permutation matrices) to tensors (that is, hypermatrices) is discussed in \cite{Csima1970} (see also, more recently, \cite{CuiLiNg04}). It was shown that the polytope of multistochastic hypermatrices in general has other vertices than those corresponding to permutation hypermatrices. 

\begin{example}{\rm The following is a $3\times 3\times 3$ ASHM:
\begin{equation}\label{eq:one}
L_3:=\left[\begin{array}{ccc}
1&0&0\\
0&1&0\\
0&0&1\end{array}\right]\nearrow
\left[\begin{array}{rrr}
0&1&0\\
1&-1&1\\
0&1&0\end{array}\right]\nearrow
\left[\begin{array}{ccc}
0&0&1\\
0&1&0\\
1&0&0\end{array}\right].\end{equation}
We get another ASHM by interchanging the top and bottom planes of $L_3$.
It is easy to verify that every $3\times 3\times 3$ PASHM is an ASHM. 
An example of an PASHM which is not an ASHM is
\begin{equation}\label{eq:ASHM}
\left[\begin{array}{r|r|r|r}
1&&& \\ \hline
&1&& \\ \hline
&&1& \\ \hline
&&& 1\end{array}\right]\nearrow
\left[\begin{array}{r|r|r|r}
&&1& \\ \hline
1&&-1&1 \\ \hline
&1&& \\ \hline
&&1& \end{array}\right]\nearrow
\left[\begin{array}{r|r|r|r}
&1&& \\ \hline
&&1& \\ \hline
&&&1 \\ \hline
1&&& \end{array}\right]\nearrow
\left[\begin{array}{r|r|r|r}
&&&1 \\ \hline
&&1& \\ \hline
1&&& \\ \hline
&1&& \end{array}\right].\end{equation}
}
\end{example} \endproof

 An $n\times n\times n$  permutation hypermatrix $A=[a_{ijk}]$ 
is ``equivalent'' to 
an $n\times n$ latin square $L_A=[l_{ij}]$. This equivalence results 
by setting $l_{ij}=k$ if and only if  $a_{ijk}=1$. 

 For example,
the $3\times 3\times 3$  permutation hypermatrix
\[\left[\begin{array}{ccc}
1&0&0\\
0&1&0\\
0&0&1\end{array}\right]\nearrow
\left[\begin{array}{ccc}
0&1&0\\
0&0&1\\
1&0&0\end{array}\right]\nearrow
\left[\begin{array}{ccc}
0&0&1\\
1&0&0\\
0&1&0\end{array}\right]\]
gives the latin square
\[\left[\begin{array}{ccc}
1&2&3\\ 3&1&2\\ 2&3&1\end{array}\right].\]

Another way to view this equivalence is as follows.    
Consider the $n\times n$ matrix $J_n$ of all 1's and a decomposition of $J_n$ into 
 $n$ permutation matrices
\[J_n=P_1+P_2+\cdots+P_n.\]
Then $P_1\nearrow P_2\nearrow\cdots\nearrow P_n$ is an $n\times n\times n$ permutation hypermatrix  and 
\[1P_1+2P_2+\cdots+nP_n\]
is a latin square; every $n\times n\times n$ permutation hypermatrix and  every $n\times n$ latin square arises in this way.

Let $A=[A_1,A_2,\ldots,A_n]$ be an $n\times n\times n$ PASHM where, therefore, $A_1,A_2,\ldots,A_n$ are $n\times n$ ASMs.
Since all vertical line sums of $A$ equal 1, it follows that
\[J_n=A_1+A_2+\cdots+A_n.\]
Let
\[L(A)=1A_1+2A_2+\cdots+nA_n.\]
Then in view of the above discussion, we call $L(A)$ a {\it PASHM-Latin square}, shortened
PASHM-LS. If $A$ is an ASHM, we call $L(A)$ an  {\it ASHM-Latin square}, shortened to  ASHM-LS. Thus $n\times n$ PASHM-LS's result from certain decompositions of $J_n$ into $n$\  $n\times n$ ASMs. Ordinary latin squares results
 in this way when $A$ is a  permutation hypermatrix, that is, when
$A_1,A_2,\ldots,A_n$ are $n\times n$ permutation matrices.
So $n\times n\times n$ ASHMs and PASHMs can also be regarded as  generalizations 
of $n\times n$ latin squares.

Problems concering hypermatrices tend to be more difficult than similar problems for matrices. As an example of this, consider the classical K\"{o}nig's minmax theorem for $(0,1)$-matrices (\cite{BR91}),  which  asserts that the term rank of a matrix $A$ (defined as the maximum number of nonzeros in $A$, no two in the same line) equals the minimum number of lines needed to cover all the nonzeros in $A$. The  term rank may be  found in polynomial time by computing a maximum matching in the corresponding bipartite graph. 
Now, for an  $n \times n \times n$ hypermatrix $A$, define the {\em term rank} of $A$ as the maximum number of nonzeros in $A$, no two of which are on a common line.

\begin{theorem}
 \label{thm:complexity}
 Finding the term rank of a hypermatrix is {\em NP}-hard.
\end{theorem}
\begin{proof}
 This follows from the fact that the 3-dimensional matching problem, denoted by 3D-MATCH, is {\em NP}-hard (\cite{GareyJohnson}). In  3D-MATCH there are given three disjoint sets $I$, $J$ and $K$, each of cardinality $n$, and a set $S \subseteq I \times J \times K$. A subset $M \subseteq S$ is a 3-dimensional matching if any two distinct elements $(i,j,k)$ and $(i',j',k')$ in $M$ satisfy $i\not = i'$, $j\not = j'$ and $k\not = k'$. Then  3D-MATCH asks for the largest size of a 3-dimensional matching. This problem reduces to finding the term rank in the $n \times n \times n$ hypermatrix $A=[a_{ijk}]$ defined by $a_{ijk}=1$ whenever $(i,j,k) \in S$.  Therefore, computing the term rank of a hypermatrix is {\em NP}-hard.
\end{proof}

\section{Diamond ASHMs}
 \label{sec:diamond}

 A {\it convex ASM} (called {\it dense} in \cite{BS16}) is defined to be an ASM having  the property that there are not any zeros between nonzeros  in both rows and columns. There is a family  ${\mathcal F}_n=\{F_n^1,F_n^2,\ldots,$ $F_n^n\}$ of $n\times n$ convex  ASMs where $F_n^k$ is defined in the following way: There are $+1$'s in the positions in the stripes of $F_n^k$ running from position $(1,k)$ to position $(k,1)$, from $(k,1)$ to position $(n,n-k+1)$, from $(n,n-k+1)$ to position $(n+1-k,n)$, and from $(n+1-k,n)$ back to position $(1,k)$. Each of the entries within the region bordered by these four stripes is nonzero and so is uniquely determined, and all positions outside of this region are zero. In particular, $F_n^1$ is the $n\times n$ identity matrix $I_n$ and $F_n^n$ is the permutation matrix $L_n$ with 1's on the back diagonal (running from position $(1,n)$ to $(n,1)$). For example,
\[
F_6^3=\left[\begin{array}{r|r|r|r|r|r}
&&1&&&\\ \hline
&1&-1&1&&\\ \hline
1&-1&1&-1&1&\\ \hline
&1&-1&1&-1&1\\ \hline
&&1&-1&1&\\ \hline
&&&1&&\end{array}\right].
\]
If $n$ is odd, then $F_n^{(n+1)/2}$ is the {\it diamond ASM} of odd order; if $n$ is even, then $F_n^{n/2}$ and $F_n^{(n+2)/2}$ are the {\it diamond ASMs} of even order.  The ASM $F_6^3$ above is a diamond ASM.

For a $(0, \pm 1)$-matrix $B$, let $\sigma_+(B)$ equal the number of its $+1$'s and $\sigma_-(B)$ the number of its $-1$'s. Then $\sigma(B)=\sigma_+(B)+\sigma_-(B)$ is the number of nonzeros of $B$.
If $2k\le n$, we have
\[
\sigma_+(F_n^k)=k(n-k+1) \;\mbox{ and } \; \sigma_-(F_n^k)=(k-1)(n-k),
\]
and hence
\[\sigma(F_n^k)=(n-k+1)k+(n-k)(k-1)=2k(n-k+1)-n.\]
The difference between the number of nonzeros of $F_n^{k+1}$ and those of  $F_n^{k}$ equals $2(n-2k)$ for $2k< n$.
 The difference between the number of $-1$'s of $F_n^{k+1}$ and those of $F_n^{k}$ equals $(n-2k)$ for $2k< n$.

We define an ASHM to be {\it convex} provided that there are not any zeros between nonzeros in any of its lines.  Some special $n\times n\times n$ convex ASHMs are the {\it diamond ASHMs} defined by
\[{\mathfrak D}_n=F_n^1\nearrow F_n^2\nearrow \cdots \nearrow F_n^n \;
\mbox{ and } \;
F_n^n\nearrow F_n^{n-1}\nearrow \cdots \nearrow F_n^1,\]
both denoted by ${\mathfrak D}_n$ for convenience. It is easy to check that 
${\mathfrak D}_n$  is indeed an ASHM. 

\begin{example}
\label{ex:diamond5}
{\rm
The diamond ASHM ${\mathfrak D}_5$ is given by
\[\left[\begin{array}{r|r|r|r|r}
1&&&&\\ \hline
&1&&&\\ \hline
&&1&&\\ \hline
&&&1&\\ \hline
&&&&1\end{array}\right]\nearrow
\left[\begin{array}{r|r|r|r|r}
&1&&&\\ \hline
1&-1&1&&\\ \hline
&1&-1&1&\\ \hline
&&1&-1&1\\ \hline
&&&1&\end{array}\right]\nearrow
\left[\begin{array}{r|r|r|r|r}
&&1&&\\ \hline
&1&-1&1&\\ \hline
1&-1&1&-1&1\\ \hline
&1&-1&1&\\ \hline
&&1&&\end{array}\right]\nearrow\]
\[\left[\begin{array}{r|r|r|r|r}
&&&1&\\ \hline
&&1&-1&1\\ \hline
&1&-1&1&\\ \hline
1&-1&1&&\\ \hline
&1&&&\end{array}\right]\nearrow
\left[\begin{array}{r|r|r|r|r}
&&&&1\\ \hline
&&&1&\\ \hline
&&1&&\\ \hline
&1&&&\\ \hline
1&&&&\end{array}\right].\]

}\end{example}  \endproof
\smallskip

 For each $k=1,2,\ldots,n$, the matrix $F_n^1+F_n^2+\cdots+F_n^k$ is an $n\times n$ $(0,1)$-matrix with exactly $k$ ones in each row and column; in particular, $F_n^1+F_n^2+\cdots+F_n^k$  has $n$ more ones than $F_n^1+F_n^2+\cdots+F_n^{k-1}$. In fact, this is a property shared by all $n\times n\times n$ ASHMs.

\begin{lemma}
\label{lem:shared}
Let $A=[A_1,A_2,\ldots, A_n]$ be an ASHM and, for each $k=1,2,\ldots,n$, let $A^{(k)}=A_1+A_2  +\cdots +A_k$. Then each $A^{(k)}$ is a $(0,1)$-matrix with $k$ ones in every row and column. In particular, $A^{(k+1)}$ has exactly $n$ more $1$'s than $A^{(k)}$ for $k=1,2,\ldots,n-1$.
\end{lemma}

\begin{proof} 
Since $A$ is an ASHM, each $A^{(k)}$ is a $(0,1)$-matrix. Every line sum in $A_i$ is 1, and therefore every line sum in $A^{(k)}=A_1+A_2  +\cdots +A_k$ is $k$. So, $A^{(k)}$ has $kn$ ones, and the last statement follows.
\end{proof}

The number of nonzeros in ${\mathfrak D}_n$ is computed to be
\[\sum_{k=1}^n\left(2k(n+1-k)-n\right)=\frac{n(n^2+2)}{3}.\]
There are other ASHMs with the same number of nonzeros. For example, if $n=4$, the ASHM
\[\left[\begin{array}{r|r|r|r}
1&&&\\ \hline
&&1&\\ \hline
&1&&\\ \hline
&&&1\end{array}\right]\nearrow
\left[\begin{array}{r|r|r|r}
&&1&\\ \hline
&1&-1&1\\ \hline
1&-1&1&\\ \hline
&1&&\end{array}\right]\nearrow
\left[\begin{array}{r|r|r|r}
&1&&\\ \hline
1&-1&1&\\ \hline
&1&-1&1\\ \hline
&&1&\end{array}\right]\nearrow
\left[\begin{array}{r|r|r|r}
&&&1\\ \hline
&1&&\\ \hline
&&1&\\ \hline
1&&&\end{array}\right]\]
has the same number 24 of nonzeros as ${\mathfrak D}_4$.

In \cite{BKMS13} it was shown that the $n\times n$ diamond ASMs have the largest number of nonzeros among all the $n\times n$ ASMs. In the next theorem we show the corresponding property for ASHMs.
Recall that if $i\le \frac{n+1}{2}$, the number of nonzeros in the $i$th row (or column) of an $n\times n$ ASM is at most $2i-1$. Since an ASM remains an ASM when the order of its rows are reversed (row $i$ becomes row $(n+1-i)$ for $1\le i\le n$) a similar inequality holds for rows $n,n-1,\ldots,\frac{n+1}{2}$. Equality holds in all these inequalities for diamond ASMs. 

Consider an $n\times n\times n$ diamond ASHM ${\mathfrak D}_n=[F_n^1,F_n^2,\ldots,F_n^n]$. The ASMs $F_n^1,F_n^2,\ldots,F_n^n$ are the horizontal planes of ${\mathfrak D}_n$.
The column-vertical planes of ${\mathfrak D}_n$ are  also $F_n^1,F_n^2,\ldots,F_n^n$ as are the row-vertical planes. Thus, both give the diamond  ASHM ${\mathfrak D}_n$. This symmetry of ${\mathfrak D}_n$ is the hypermatrix analogue of the symmetry of the diamond ASMs.  The $i$th row of the $k$th horizontal plane $F_n^k$ is also the $k$th row of the $i$th row-vertical plane $F_n^i$ of ${\mathfrak D}_n$, and hence the number of its nonzeros  is $2\min\{i,k\}-1$. 

\begin{theorem}
\label{th:diamond}
The maximum number $m_n$ of nonzeros in an $n\times n\times n$  ASHM is given by
\begin{equation}\label{eq:number}m_n=\frac{n(n^2+2)}{3}\end{equation}
and this maximum is attained by the diamond ASHM ${\mathfrak D}_n$. In fact, we have the following:
\begin{itemize}
\item[\rm (i)] $F_n^k$ has the largest number of nonzeros among all $n\times  n$ 
ASMs $A$ such that there exist ASMs $A_1,\ldots,A_{k-1},A_{k+1},\ldots, A_n$ for which $[A_1,\ldots, A_{k-1},A, A_{k+1},\ldots,A_n]$  is an ASHM.
\item[\rm (ii)] More generally, the number of nonzeros in a row $($resp., column$)$ of such an $A$ is at most equal to the number of nonzeros in the corresponding row $($resp., column$)$ of $F_n^k$.
\end{itemize}
\end{theorem}

\begin{proof} Let $A=[A_1,A_2,\ldots,A_n]$ be an $n\times n\times n$ ASHM. Then row $i$ of the ASM $A_k$ is also row $k$ of the $i$th row-vertical ASM plane of $A$. Hence the number of nonzeros in row $i$ of $A_k$ cannot exceed the minimum of the numbers of nonzeros possible in row $i$  and row $k$ of  $n\times n$ ASMs, that is, cannot exceed $2\min\{i,k\}-1$. But this is the number of nonzeros in row $i$ of the $k$th horizontal ASM plane of  the   diamond ASHM ${\mathfrak D}_n$. Since this is true for all $i$ and $k$,  the conclusions (i) and (ii) in the theorem now follow.
\end{proof}

We remark that this proof uses strongly that we have an ASHM, not just a PASHM.

\section{ASHM-Latin squares}
\label{sec:latin}

In this section we investigate ASHM-LS's and PASHM-LS's. 

\begin{example}\label{ex:new}{\rm 
The PASHM in (\ref{eq:ASHM}) gives the PASHM-LS
\[\left[\begin{array}{cccc}
1&3&2&4\\
2&1&5&2\\
4&2&1&3\\
3&4&2&1\end{array}\right].\]

 Consider the $4\times 4\times 4$ diamond ASHM ${\mathfrak D}_4=[F_4^1,F_4^2,F_4^3,F_4^4]$ given by
\[
\left[\begin{array}{rrrr}
1&0&0&0\\
0&1&0&0\\
0&0&1&0\\
0&0&0&1\end{array}\right]\nearrow
\left[\begin{array}{rrrr}
0&1&0&0\\
1&-1&1&0\\
0&1&-1&1\\
0&0&1&0\end{array}\right]\nearrow
\left[\begin{array}{rrrr}
0&0&1&0\\
0&1&-1&1\\
1&-1&1&0\\
0&1&0&0\end{array}\right]\nearrow
\left[\begin{array}{rrrr}
0&0&0&1\\
0&0&1&0\\
0&1&0&0\\
1&0&0&0\end{array}\right].\]
Then the corresponding ASHM-LS is
\[L({\mathfrak D}_4)=
\left[\begin{array}{cccc}
1&2&3&4\\
2&2&3&3\\
3&3&2&2\\
4&3&2&1\end{array}\right].\]
Note that the hypermatrix in this example has the property that viewed in each of the three directions  the same ASHM results and thus the same ASHM-LS results.
This is true for all diamond ASHMs. \endproof
}
\end{example}

\begin{lemma}
 \label{lem:one} 
 Let
$A=[a_{ijk}]=[A_1,A_2,\ldots,A_n]$ be an $n\times n\times n$ PASHM with corresponding $n\times n$ PASHM-LS\, $L(A)=1A_1+2A_2+\cdots+nA_n=[l_{ij}]$. Then
 the sum of the entries in each row and column of $L(A)$ equals
${n+1}\choose 2$.
If $A$ is an ASHM then, in addition,  the following hold:
\begin{itemize}
\item[\rm (i)] The set of entries of $L(A)$ is $\{1,2,\ldots,n\}$, and  the first and last rows and
columns are permutations of $1,2,\ldots,n$, 
\item[\rm (ii)] Let $i, j \le n$, and let
$a_{ijk}=\pm 1 \mbox{ for } k\in\{k_1,k_2,\ldots,k_p\}$, where $1\le p\le n$ and  $1\le k_1< k_2<\cdots< k_p\le n$ and  where $a_{ijk}=0$,
 otherwise. 
 \begin{itemize}
\item[\rm (iia)] If $l_{ij}=r$, then $k_1\le r$, and $|\{k:a_{ijk}\ne 0\}|\le 2r-1$; moreover, $k_1=r$ implies that $p=1$.
 In particular, if $l_{ij}=1$, then $a_{ij1}=1$ and $a_{ijk}=0$ for $1<k\le n$.
\item[\rm (iib)] If $l_{ij}=r$, then $|\{k:a_{ijk}\ne 0\}|\le 2(n-r)+1$.
In particular, if $l_{ij}=n$
then $a_{ijn}=1$ and $a_{ijk}=0$ for $1\le k<n$. 
\end{itemize}
\end{itemize}
\end{lemma}

\begin{proof} 
Consider some row $i$ of $L(A)$ (similar arguments work for a column). Since $A_1,A_2,\ldots,A_n$ are ASMs, row $i$ of each  $A_i$ sums to 1. Thus   the sum of the entries of $L(A)$ in row $i$ equals
\[\sum_{k=1}^n k\left(\sum_{j=1}^n a_{ijk}\right)=\sum_{k=1}^n k\cdot 1={{n+1}\choose 2}.\]

Now assume that $A$ is an ASHM. 
As in (ii),
let $i, j \le n$, and let
$a_{ijk}=\pm 1 \mbox{ for } k\in\{k_1,k_2,\ldots,k_p\}$, where $1\le p\le n$ and  $1\le k_1< k_2<\cdots< k_p\le n$ and  where $a_{ijk}=0$,
 otherwise.
Then $p$ is odd and $a_{ijk}=1$ for $k=k_1,k_3,\ldots,k_{p}$ and $a_{ijk}=-1$ for $k=k_2,k_4,\ldots,k_{p-1}$. Thus
\begin{equation}\label{eq:better}
l_{ij}=k_1+(-k_2+k_3)+(-k_4+k_5)+\cdots + (-k_{p-1}+k_p)\ge k_1+\frac{p-1}{2}\ge k_1\ge 1,\end{equation}
 and also
\begin{equation}\label{eq:better2}
l_{ij}=(k_1-k_2)+(k_3-k_4)+\cdots + (k_{p-2}-k_{p-1})+k_p\le k_p-\frac{p-1}{2}\le k_p\le n.\end{equation}
Thus $1\le l_{ij}\le n$ for each $i,j$.
That the first and last rows and
columns are permutations of $1,2,\ldots,n$ follows from the definitions of  an ASHM and ASHM-LS. This proves (i). Assertions (iia) and (iib) follow easily from equations (\ref{eq:better}) and (\ref{eq:better2}).
\end{proof}

The first example in Example \ref{ex:new} shows that assertion (iia) in Lemma \ref{lem:one} does not hold in general for PASHM-LS's.

By Lemma \ref{lem:one}, and as is the case for $n\times n$ LS's, the entries of an $n\times n$ ASHM-LS are $\{1,2,\ldots,n\}$ but, unlike for LS's,  repeats in a row or column are possible.
As with LS's, we can regard an $n\times n$ ASHM-LS $L$ as an $n\times n\times n$ $(0,1)$-hypermatrix
\begin{equation}\label{eq:dec1}
H_L: Q_1\nearrow Q_2\nearrow\cdots\nearrow Q_n\end{equation}
where an entry $k$ in the $(i,j)$-position of $L$ becomes a $1$ in the $(i,j,k)$-position of $H_L$ giving $n\times n$ $(0,1)$-matrices $Q_1,Q_2,\ldots, Q_n$.
Since the entries of $H_L$ are $\{1,2,\ldots,n\}$, it follows that
\begin{equation}\label{eq:dec2}
Q_1+Q_2+\cdots+Q_n=J_n,\end{equation} a decomposition of $J_n$ into $n\times n$ $(0,1)$-matrices. We call  (\ref{eq:dec2}) (and  (\ref{eq:dec1})) the {\it $(0,1)$-decomposition}
of the ASHM-LS $L$. 

In summary, with the above notation, for the $n\times n\times n$ ASHM $A=[A_1,A_2,\ldots,A_n]$ with corresponding latin square specified by $L=[Q_1, Q_2,\ldots,Q_n]$, we have the two decompositions of $J_n$ given by
\[J_n=Q_1+Q_2+\cdots+Q_n=A_1+A_2+\cdots+A_n.\]

\begin{example}\label{ex:six}{\rm
The $4\times 4$ ASHM-LS $L$ in Example \ref{ex:new}
gives
\[H_L=
\left[\begin{array}{cccc}
1&0&0&0\\  
0&0&0&0\\  
0&0&0&0\\  
0&0&0&1\end{array}\right]\nearrow
\left[\begin{array}{cccc}
0&1&0&0\\  
1&1&0&0\\  
0&0&1&1\\  
0&0&1&0\end{array}\right]\nearrow
\left[\begin{array}{cccc}
0&0&1&0\\  
0&0&1&1\\  
1&1&0&0\\  
0&1&0&0\end{array}\right]\nearrow
\left[\begin{array}{cccc}
0&0&0&1\\  
0&0&0&0\\  
0&0&0&0\\ 
1&0&0&0\end{array}\right].\]

}\end{example} \endproof

The next theorem asserts that an ASHM-LS which is an ordinary LS can only arise in the classical way.

\begin{theorem}\label{th:LS}
Let $A=[a_{ijk}]$ be an $n\times n\times n$ ASHM. Suppose the ASHM-LS $L(A)$ is  a latin square. Then $A$ is a permutation hypermatrix.
\end{theorem}

\begin{proof} Let $A=[A_1,A_2,\ldots,A_n]$.
Since $L(A)=[l_{ij}]$ is a latin square,  it follows that for each $r=1,2,\ldots,n$, the positions of the $n$ $r$'s in $L(A)$  are those occupied by 1's in an $n\times n$ permutation matrix $P_r$.
From equation (\ref{eq:better}) in the proof of Lemma \ref{lem:one} we conclude that the number of nonzeros in each  column $C_{ij}$ of positions
$\{(i,j,k):1\le k\le n\}$ with $l_{ij}=1$ equals one. It also follows from (\ref{eq:better})  that $(i,j,k)\in C_{ij}$ implies that $k=1$. Thus $A_1$ is an $n\times n$ permutation matrix whose 1's are in those positions occupied by 1's in $L(A)$, and $A_2,\ldots,A_n$ have 0's in those positions. Since in $L(A)$, the positions occupied by $2$'s are zeros in $A_1$, we can use a similar argument to show that  $A_2$ is an $n\times n$  permutation matrix whose 1's
are in those positions occupied by 2's in $L(A)$, and $A_3,\ldots,A_n$ have 0's in those positions.
Proceeding inductively, we conclude that $A$ is a permutation hypermatrix.
\end{proof}

\medskip
Recall that,  for  vectors $x=(x_1, x_2, \ldots, x_n)$ and 
$y=(y_1, y_2, \ldots, y_n)$,  $x$ is {\em majorized} by $y$, written $x \preceq y$, whenever $\sum_{j=1}^k x_{[j]} \le \sum_{j=1}^k y_{[j]}$ ($k \le n$) with equality  for $k=n$. Here $x_{[j]}$ denotes the $j$th largest component in $x$. We now define a new order for matrices based on majorization. Let $A$ and $B$ be $m \times n$ matrices. We say that $A$ is {\em line majorized} by $B$ if each line (row or column) of $A$ is majorized by the corresponding line in $B$, and then we write $A \preceq_l B$. This is a preorder on the class of $m \times n$ matrices. 
Let $\Row_i(A)$ denote the $i$th row of a matrix $A$.

\begin{theorem}
 \label{thm:ASHM-majorization}
  Let $A$ be an $n\times n\times n$ ASHM and $P$ an $n\times n\times n$ permutation ASHM. Then
  \begin{equation}
   \label{eq:l(A)-maj}
       L(A) \preceq_l L(P),
  \end{equation}
that is, each row of the ASHM-LS $L(A)$ is majorized by $z_n=(n,n-1,\ldots,1)$.
\end{theorem}
\begin{proof}
First note that each row of $L(P)$ is a permutation of  $z_n=(n, n-1,  \ldots, 1)$ (recall, $n$-tuples are identified with column vectors).
 Let $i \le n$. Let $C$ be the $i$th row-vertical-plane of $A$ so that $C$ is an ASM.  The $i$th row in $L(A)$ satisfies 
\[
    \Row_i(L(A))=z_n^T C.
\]
(where we view the matrix $C$ so the layers of $A$ are organized downwards). 
Let $B=C^T$, so $B$ is also an ASM, and define $w=Bz_n=(z_n^TC)^T$. We  need to show the majorization
 \[
       (*)\;\; \;\; w \preceq z_n.
\]
 Let $e^{(k)}=e_1+e_2+ \cdots + e_k$  ($k \le n$) where $e_j$ is the $j$th unit vector. So 
\[
    z_n=\sum_{k=1}^n e^{(k)}.
\]
 Then
\begin{equation}
 \label{eq:sum-vk}
   w=Bz_n=B\sum_{k=1}^n e^{(k)}= \sum_{k=1}^n Be^{(k)}=\sum_{k=1}^n v^{(k)}
\end{equation}
where $v^{(k)}=Be^{(k)}$  ($k \le n$). Then $v^{(k)}$ is the sum of the first $k$ columns of  $B$, so $v^{(k)}$ is a $(0,1)$-vector, as $B$ is an ASM. Moreover, $v^{(k)}$ has $k$ ones.
Let $V$ be the $n \times n$ matrix whose columns are $v^{(1)}, v^{(2)}, \ldots, v^{(n)}$. Then $V$ is a $(0,1)$-matrix with column sum vector $(1,2, \ldots, n)$ and row sum vector $w$, by (\ref{eq:sum-vk}). We may therefore apply the (simple part of the) Gale-Ryser theorem, and conclude that $w$ is majorized by the conjugate $(1, 2, \ldots, n)^*$ of the column sum vector $(1, 2, \ldots, n)$. But $(1, 2, \ldots, n)^*=z_n$, so we have shown that $w \preceq z_n$, and hence the  majorization in $(*)$ holds. Similarly, one shows the majorization for columns of $L(A)$ and $L(P)$. 
\end{proof}

Consider again the diamond ASHM ${\mathfrak D}_n$. The matrix $L({\mathfrak D}_n)$ remains the same if its rows and also its columns are simultaneously reversed. Moreover, when 
$i\le \lfloor n/2\rfloor$, then its  $i$th row  is 
\[
    \Row_i(L({\mathfrak D}_n))=v^{(i)}:=(i,i, \ldots, i, i+1,i+2, \ldots, n-i+1, n-i+1, \ldots, n-i+1)
\]
where the first (resp. last) $i$ entries are equal.

\begin{example}
\label{ex:diamond7}
{\rm
Let $n=7$. Then 

\[
L({\mathfrak D}_7)=
\left[
\begin{array}{ccccccc}
     1  &   2  &   3  &   4  &   5  &   6  &   7 \\
     2  &   2  &  3   &  4   &  5   &  6  &   6\\
     3   &  3   &  3  &   4  &   5  &   5   &  5\\
     4   &  4  &   4  &   4   &  4  &   4  &   4\\
     5  &   5   &  5  &   4  &   3  &   3  &   3\\
     6   &  6   &  5  &   4  &   3  &   2  &   2\\
     7  &   6  &   5  &   4  &   3  &   2  &   1
\end{array}
\right].
\]
Consider  an ASHM $A$, whose first three horizontal planes coincide with those of ${\mathfrak D}_7$ while the remaining planes are suitable permutation matrices $P_4,P_5,P_6, P_7$. The sum of these three horizontal planes is
\[S=\left[\begin{array}{ccccccc}
1&1&1&0&0&0&0\\
1&1&0&1&0&0&0\\
1&0&1&0&1&0&0\\
0&1&0&1&0&1&0\\
0&0&1&0&1&0&1\\
0&0&0&1&0&1&1\\
0&0&0&0&1&1&1\end{array}\right].\]
Then $P_4+P_5+P_6+P_7=J_7-S$. By specifically  choosing  $P_4$ to be the permutation matrix
\[P_4=\left[\begin{array}{c|c|c|c|c|c|c}
&&&&&&1\\ \hline
&&1&&&&\\ \hline
&1&&&&&\\ \hline
1&&&&&&\\ \hline
&&&&&1&\\ \hline
&&&&1&&\\ \hline
&&&1&&&\end{array}\right],\]
we can obtain
\begin{equation}\label{eq:n=7}
L(A)=
\left[
\begin{array}{ccccccc}
     1  &   2  &   3 &    7   &  5   &  6  &   4\\
     2  &   2   &  3  &   3   &  6   &  7  &   5\\
     3   &  3  &   2 &    5   &  3   &  5   &  7\\
     4  &   3  &   4  &   2   &  6   &  3  &   6\\
     6  &   7  &   3  &   4   &  2   &  3  &   3\\
     7  &   5  &   6  &   3   &  3   &  2  &   2\\
     5  &   6  &   7  &   4   &  3   &  2  &   1
\end{array}
\right].
\end{equation}
For instance, concerning row four in $L(A)$  we have the majorization 
\[
  (4,  3,  4,  2,   6,  3,  6) \preceq z_7=(7,6,5,4,3,2,1)
\]
which is in accordance with Theorem \ref{thm:ASHM-majorization}. In this example, we can also check that every row or column of $L(A)$ majorizes the corresponding row or column of $L({\mathfrak D}_7)$, so $L({\mathfrak D}_7)  \preceq_l L(A)$. 
\endproof
} 
\end{example} 

In an $n\times n$  LS, each integer in $\{1,2,\ldots,n\}$ occurs exactly $n$ times. In an $n\times n$ ASHM-LS, the entries are also taken from $\{1,2,\ldots,n\}$, but their multiplicity can vary. In the $7\times 7$ ASHM-LS in (\ref{eq:n=7}), the integer $3$ occurs 14 times. This example can be generalized to all odd $n\ge 5$ giving an $n\times n$ ASHM-LS in which $(n-1)/2$ occurs $2n$ times.

In view of Example \ref{ex:diamond7} one may ask if $L({\mathfrak D}_n)  \preceq_l L(A)$ holds for every ASHM $A$ of size $n \times n \times n$.  As shown by the next example, this is not the case,

\begin{example}
\label{ex:diamond5more}
{\rm
Consider the ASHM $A'$ given by 
\[\left[\begin{array}{r|r|r|r|r}
1&&&&\\ \hline
&1&&&\\ \hline
&&&1&\\ \hline
&&1&&\\ \hline
&&&&1\end{array}\right]\nearrow
\left[\begin{array}{r|r|r|r|r}
&&1&&\\ \hline
&&&1&\\ \hline
1&&&-1&1\\ \hline
&1&-1&1&\\ \hline
&&1&&\end{array}\right]\nearrow
\left[\begin{array}{r|r|r|r|r}
&1&&&\\ \hline
&&1&&\\ \hline
&&&1&\\ \hline
&&&&1\\ \hline
1&&&&\end{array}\right]\nearrow\]
\[\left[\begin{array}{r|r|r|r|r}
&&&&1\\ \hline
1&&&&\\ \hline
&1&&&\\ \hline
&&1&&\\ \hline
&&&1&\end{array}\right]\nearrow
\left[\begin{array}{r|r|r|r|r}
&&&1&\\ \hline
&&&&1\\ \hline
&&1&&\\ \hline
1&&&&\\ \hline
&1&&&\end{array}\right].\]
\smallskip
Let $A$ be the ASHM obtained from $A'$ by letting the horizontal planes of $A'$ become the row-vertical planes of $A$. Let $x$ denote the  second row of the ASHM-LS $L(A)$ associated with $A$. Then $x$ is computed from the second matrix above (the only one with some negative entries) and we get 
\[
   x=(3,2,4,3,3).
\]
Consider the diamond ${\mathfrak D}_5$.  Then (see above) the second row of $L({\mathfrak D}_5)$ is $(2,2,3,4,4)$. But $(2,2,3,4,4) \not \preceq (3,2,4,3,3)$ because $4+4>4+3$. \endproof
}
\end{example}  

We say that a line in a matrix is {\em constant} if all its entries are equal. 
Let $A=[A_1, A_2,\ldots,A_n]$ be an $n \times n \times n$ ASHM and let $L=L(A)$ be its   ASHM-LS. Assume that  $L$ has a constant line. Then $n$ must be odd. In fact,  by Lemma \ref{lem:one}, the sum of the entries in any line of $L$ is $n(n+1)/2$, so in the line which is constant each entry must be equal to $(n+1)/2$. But every entry in $L$ is integral, so $n$ must be odd. 

For instance, for the  diamond ASHM ${\mathfrak D}_n$, where $n$ is odd, say $n=2k-1$, the $k$'th row and column of the ASHM-LS  $L({\mathfrak D}_n)$ are constant lines. The case of $n=7$ is  shown in Example \ref{ex:diamond7}. 

\begin{example}
\label{ex:constant-line}
{\rm
Consider the ASM $B$ be given by 
\[
B=
\left[\begin{array}{r|r|r|r|r|r|r}
&&&&1&\\ \hline
&1&&& & \\ \hline
&&1&&-1&1\\ \hline
1&-1&&1&&-1&1\\ \hline
&&&&&1\\ \hline
&1&-1&&1&\\ \hline
&&1&&&\end{array}
\right].
\]
Then $z_7^TB=(4,4,4,4,4,4,4)$. It is possible to construct an ASHM $A$ such that $B$ is the third row-vertical plane; this follows from results we establish in the next section. Thus, the third row of the corresponding ASHM-LS $L(A)$ is $(4,4, \ldots, 4)$. Note that this vector is minimal in the majorization order in $\mb{R}^7$. On the other hand, the first row of $L(A)$ is a permutation of $(1,2, \ldots, 7)$ which, by Theorem \ref{thm:ASHM-majorization}, is maximal in the majorization order among all possible rows of an ASHM-LS of order $7$. \endproof
}
\end{example}   

The ASHM-LS of a permutation hypermatrix does  not contain any  constant line, except in the trivial case of $n=1$. For $n=3$, the ASHM-LS $L({\mathfrak D}_3)$ has a constant line, both a row and a column. An ASHM  where every line  contains  at most one $-1$ will be called a {\em near-permutation hypermatrix}. We now show that near-permutation ASHMs have this same property as permutation ASHMs, namely the corresponding ASHM-LS's do not have any constant line (when $n>3$).

\begin{theorem}
 \label{thm:constant_not}
 Let $A=[A_1, A_2, \ldots, A_n]$ be an $n \times n \times n$ near-permutation ASHM where $n>3$. Then the  ASHM-LS $L(A)$ does not have any constant lines. 
\end{theorem}
\begin{proof}
 Assume $L=L(A)$ has a constant line, say row $i$. As explained above,  $n$ must be odd, so  $n=2k-1$ with $k\ge 3$, and every entry in row $i$ of $L$ must be $(n+1)/2=k$. Let $B=[b_{ij}]$ be the row-vertical plane of $A$ associated with row $i$ of $L$; this row equals $z_n^TB$ and all entries are $k$. Then $B$  cannot be a permutation matrix, because then row $i$ of  $L$ is a permutation of $(1,2, \ldots, n)$. So,  some line in $B$ has a negative entry. On the other hand, no line in $B$ can have more than one negative entry, as $A$ is a near-permutation ASHM. The first column of $B$ is a unit vector, so its 1 must be in row $k$, i.e.,  $b_{k1}=1$. This implies that the  second column cannot be a unit vector, because then its 1 would be in another row than the  $k$th, and then $z$ times that column would be different from $k$. So, the second column of $B$ contains a $-1$, which then must be in row $k$,  again by the alternating property of $B$. By similar arguments, $b_{kn}=1$ and $b_{k,n-1}=-1$. However, as $n>3$, this means that row $k$ in $B$ has two $-1$'s; a contradiction. This proves that no row in $L$ is constant, and similar arguments show that no column is constant. 
\end{proof}

For ASHMs with more than one negative entry in  some lines, the situation is more complicated. However, the construction in the previous proof gives the following property.

\begin{corollary}
 \label{cor:constant_not}
 Let $A=[A_1, A_2, \ldots, A_n]$ be an $n \times n \times n$ ASHM and let $L=L(A)$ be its   ASHM-LS. If $n>3$, then the second $($resp. second-last$)$ row or column in $L$ is not constant.  
\end{corollary}
\begin{proof}
 We can use the same arguments as in the proof of Theorem \ref{thm:constant_not}. The only change is that the ASM $B$ cannot have more than one negative entry in any line, because this would contradict that $B$ is the second row-vertical plane of $A$ (as, in that case, the first row-vertical plane has two ones in the same line). 
 \end{proof}

Example \ref{ex:constant-line} shows that the {\em third} row or column of  an ASHM-LS may be constant.

Let again $z_n=(n,n-1,\ldots,2,1)$. 
The rows and columns of an $n\times n$ ASHM-LS are obtained by vector-matrix multiplications
\begin{equation}
\label{eq:projection}
 z_n ^TA=v(A)=(v_1,v_2,\ldots,v_n) \mbox{ and } A^Tz_n=h(A)=(h_1,h_2,\ldots,h_n)
\end{equation}
where $A$ is some  $n\times n$ ASM.
We call $v(A)$ in (\ref{eq:projection}) the {\it weighted vertical projection} of the ASM $A$, and $h(A)$ the {\it weighted horizontal projection}. It follows from Theorem
 \ref{thm:ASHM-majorization} that the weighted projections of an $n\times n$ ASM are majorized by $z_n$. Thus the weighted projections of an ASM can be regarded as  ``integral smoothings'' of the entries of $z_n$. As we know, the weighted vertical and horizontal projections of the diamond ASMs are the $n$-vectors
 $(\frac{n+1}{2},\frac{n+1}{2},\ldots,\frac{n+1}{2})$ if $n$ is odd and
 $(\frac{n}{2}+1,\ldots,\frac{n}{2}+1,\frac{n}{2},\ldots,\frac{n}{2})$, with $n/2$ each of the two different components, if $n$ is even.
If $A$ is a permutation matrix $P$ corresponding to the permutation $(i_1,i_2,\ldots,i_n)$ of $\{1,2,\ldots,n\}$,
then $h(P)=(i_1,i_2,\ldots,i_n)$, and $v(P)=(j_1,j_2,\ldots,j_n)$ where $(j_1,j_2,\ldots,j_n)$ is the inverse 
permutation. So either of $v(P)$ and $h(P)$ determines $P$. This is in contrast to ASMs in general where two different ASMs can have both the same  weighted horizontal projections and the same
weighted vertical projections.

\begin{example}\label{ex:projections}{\rm
Let
\[A=\left[\begin{array}{r|r|r|r|r|r|r}
&&\phantom{-1}&1&\phantom{-1}&&\\ \hline
&1&&-1&&1&\\ \hline
&&&1&&&\\ \hline
1&-1&1&-1&1&-1&1\\ \hline
&&&1&&&\\ \hline
&1&&-1&&1&\\ \hline
&&&1&&&\end{array}\right].\]
Then $A\ne D_7$ but $v(A)=h(A)=(4,4,4,4,4,4,4)=h(D_7)=v(D_7)$.\endproof
}
\end{example}

In the next lemma, we show that  if the weighted projection of an $n\times n$ ASM  $A$  is a permutation of $\{1,2,\ldots,n\}$, then $A$ is a permutation matrix. Thus, if $A$ is not a permutation matrix, then $z_n^TA$ must have a repeated entry.

\begin{lemma}\label{lem:perm-matrix}
Let $A=[a_{ij}]$ be an $n\times n$ ASM. If $z_n^TA=z_n^TP$ for some $n\times n$  permutation matrix $P$, then $A=P$.
\end{lemma}

\begin{proof} Let $v=(i_1,i_2,\ldots,i_n)=z_n^TP=z_nA$, where $v$ is a permutation of $\{1,2,\ldots,n\}$. Let $k$ be such that $i_k=1$. Since the $+1$'s and $-1$'s of $A$ alternate in the columns beginning and ending with a $+1$, it follows that column $k$ of $A$ must equal $(0,0,\ldots,0,1)$; since $A$ is an ASM, the last row of $A$ has only zeros in columns different from $k$. Thus, the last row and column $k$ of $A$ agree with the corresponding row and column of $P$.

Now let $l$ be such that $i_l=2$. Since column $l$ of $A$ has a zero in row $n$, column $l$ of $A$ cannot contain three or more nonzeros, for if it did, $u$ times column $l$ of $A$ would be more than 2. It follows that column $l$ of $A$ equals $(0,0,\ldots,0,1,0)$, and the  row $n-1$ of $A$ has only zeros in columns different from $l$. Hence $A$ agrees with $P$ in both  row $n-1$ and column $l$. Continuing like this, if follows by induction that $A=P$.
\end{proof}

\begin{corollary}\label{cor:perm-matrix}
Let $A$ be an $n\times n\times n$ ASHM with corresponding ASHM $L(A)$.  If $n-1$ rows $($resp., columns$)$ of $L(A)$ are permutations of $\{1,2,\ldots,n\}$, then $A$ is an $n\times n\times n$ permutation hypermatrix and $L(A)$ is a latin square.
\end{corollary}

\begin{proof}
By Lemma  \ref{lem:perm-matrix}, $n-1$ of the row-vertical-planes of $A$ are permutation matrices. Since $A$ is an ASHM, it follows that the remaining 
row-vertical-plane is also a permutation matrix; hence $A$ is a permutation hypermatrix and so $L$ is a latin square.
\end{proof}

We conclude this section with the following variation of the majorization order.
For vectors $x=(x_1, x_2, \ldots, x_n)$ and $y=(y_1, y_2, \ldots, y_n)$ (not assumed to be monotone), we write $x \preceq^* y$ provided that
\[
      \sum_{j=1}^p x_j \le \sum_{j=1}^p y_j \;\;\;(p \le n)
\]
where equality holds for $p=n$. 
Note that we only require inequalities to hold for the leading partial sums. 
For an $n \times n$ matrix $A$, let $h(A)$ denote its weighted horizontal projection, i.e., $h(A)=A^Tz_n$ where $z_n=(n,n-1, \ldots, 1)$.

\begin{theorem}
 \label{thm:ASM-chain}
  Let $A$ be an $n \times n$ ASM. Let $p=\sigma_- (A)$. Then there is a sequence of $n \times n$ ASMs 
  \[
     A^{(p)}, A^{(p-1)}, \ldots,  A^{(0)}
  \]
  where $(i)$ $A^{(p)}=A$, $(ii)$ $A^{(0)}$ is a permutation matrix,  $(iii)$ $\sigma_-(A^{(s)})=s$ $(s=0, 1, \dots, p)$, and $(iv)$ the following majorizations hold
  \[
     h(A^{(p)}) \preceq^* h(A^{(p-1)}) \preceq^* \cdots \preceq^* h(A^{(0)}).
  \]
\end{theorem}
 \begin{proof} 
 If $p=0$ ($A$ is a permutation matrix), there is nothing to prove, so assume $p=\sigma_-(A)\ge 1$ where $A=[a_{ij}]$ is an ASM. 
 Choose $(k,l)$ with $k+l$ minimal such that $a_{kl}=-1$. Then there is  a unique $l'<l$ with $a_{kl'}$ nonzero, and we have $a_{kl'}=1$. Similarly, there is  a unique $k'<k$ with $a_{k'l}$ nonzero and $a_{k'l}=1$. By the choice of $k, l$,  there is no negative entry in the leading $k \times l$ submatrix, apart from $a_{kl}$. Moreover, $a_{k'l'}=0$. Let $B=[b_{ij}]$ be obtained from $A$ by letting $b_{kl}=b_{k'l}=b_{kl'}=0$,  $b_{k'l'}=1$, and, otherwise, $b_{ij}=a_{ij}$. Then $B$ is an ASM and $\sigma_-(B)=\sigma_-(A)-1$. 
 
 Moreover, $B$ is obtained from $A$ by adding the submatrix 
 \[K=
  \left[
  \begin{array}{rr}
     1 & -1 \\*[\smallskipamount]
     -1 & 1
  \end{array}
  \right]
 \]
to the submatrix of $A$ corresponding to rows $k'< k$ and columns $l'<l$.  Thus, denoting the $s$'th component of $h(A)$ and $h(B)$  by $h(A)_s$ and $h(B)_s$, respectively,  
 \[
 \begin{array}{ll}
    h(B)_{l'}=h(A)_{l'}+(k-k'), \;\; h(B)_{l}=h(A)_{l}-(k-k'), \;\;\mbox{\rm and�} \\*[\smallskipamount]
  h(B)_{i}=h(A)_{i} \;\;\mbox{\rm for all other $i$.}
 \end{array} 
 \]
 This implies that 
 \[
    h(B) \preceq^* h(A). 
 \]
 We now repeat this process with $A$ replaced by $B$. Clearly, after $p$ such operations, we have produced a sequence of ASMs with all the desired properties.
\end{proof}

\begin{example}
 \label{ex:maj-chain}
 {\rm 
 Consider the following ASM $A$ with $h(A)=(2,3,3,4,3)$:
\[
A=A^{(2)} =
\left[
\begin{array}{r|r|r|r|r}
&&&+&\\ \hline
&+&&&\\ \hline
&&+&-&+\\ \hline
+&-&&+&\\ \hline
&+&&&\\ 
\end{array}
\right].
\]
Adding the matrix $K$ to the submatrix induced by rows $2, 4$ and columns $1,2$ gives 
\[
A^{(1)} =
\left[
\begin{array}{r|r|r|r|r}
&&&+&\\ \hline
+&&&&\\ \hline
&&+&-&+\\ \hline
&&&+&\\ \hline
&+&&&\\ 
\end{array}
\right] 
\]
with $w(A^{(1)})=(4,1,3,4,3)$. Finally, adding $K$ to the submatrix induced by rows $1,3$ and columns $3,4$ gives the permutation matrix
\[
A^{(0)} =
\left[
\begin{array}{r|r|r|r|r}
&&+&&\\ \hline
+&&&&\\ \hline
&&&&+\\ \hline
&&&+&\\ \hline
&+&&&\\ 
\end{array}
\right] 
\]
with $w(A^{(0)})=(4,1,5,2,3)$. Here we have the majorizations
\[
   (2,3,3,4,3)   \preceq^* (4,1,3,4,3)  \preceq^* (4,1,5,2,3). 
\] \endproof
}

\end{example}





\section{Completion problems}
 \label{sec:completion}

In this section we study completion problems where some ASMs  are given as horizontal-planes and one wants to extend them to obtain an  ASHM.

Consider a  
 $(0, \pm 1)$-vector $x=(x_1, x_2, \ldots, x_p)$. We say that $x$ is  {\em $(1,*)$-alternating} if its nonzeros (if any) alternate in sign and the first nonzero is a 1.  Similarly, $x$ is  {\em $(*,1)$-alternating} if its nonzeros (if any) alternate in sign and the last nonzero is a 1.  Thus all the lines of an ASM are both $(1,*)$-alternating and $(*,1)$-alternating.

 The $i$th row (resp. column) of a matrix $X$ is denoted by $\row_i(X)$ (resp. $\col_i(X)$).  As before the matrix $J_n$ is the all ones matrix of order $n$.

Let  $(n-1)$ $n\times n$ ASMs $A_1,\ldots, A_{k-1}, A_{k+1},\ldots, A_n$ be given, and let $A_{(k)}$ denote the corresponding $n \times n \times (n-1)$ hypermatrix.  We consider the question: When does there exist an $n\times n$ ASM $A_k$ such that
$A=[A_1,\ldots, A_{k-1}, A_k, A_{k+1},\ldots, A_n]$ is an ASHM? If the answer is affirmative, we say that  $A_{(k)}$ has an {\em ASHM-completion $A$ at $($horizontal$)$ layer $k$}. Let $A_{(k)}=[a_{ijs}]$, where $s \in \{1, 2, \ldots, n\}\setminus \{k\}$.  An obvious necessary condition for  $A_{(k)}$ to have an  ASHM-completion at layer $k$ is 
\begin{equation}
   \label{eq:completion-cond1}
    \begin{array}{llll}
   {\rm (i)} & (a_{ij1}, \ldots, a_{ij,k-1}) &\mbox{\rm is $(1,*)$-alternating} &(1 \le i,j \le n), \\*[\smallskipamount]
   {\rm (ii)} & (a_{ij,k+1}, \ldots, a_{ijn}) &\mbox{\rm is $(*,1)$-alternating} &(1 \le i,j \le n). \\
    \end{array} 
\end{equation}
For the given $A_{(k)}$,  (\ref{eq:completion-cond1}) implies that
$ \sum_{s\not = k} A_s$  is a $(0,1,2)$-matrix.

\begin{theorem}
 \label{thm:ASM-completion1}
  Let $A_1,\ldots, A_{k-1}, A_{k+1},\ldots, A_n$ be   $(n-1)$  $n\times n$ ASMs. Let $A_{(k)}$ be the corresponding $n\times n\times (n-1)$ hypermatrix and let $X^{(k)}=\sum_{s\not = k} A_s$. Define $Y=J_n-X^{(k)}$. 
    Then  $A_{(k)}$ has an ASHM-completion at layer $k$ if and only if both $(\ref{eq:completion-cond1})$ and  the following majorization conditions hold
\begin{equation}
   \label{eq:completion-cond2}
    \begin{array}{lrll}
   {\rm (i)} & \row_i(X^{(k)}) \preceq^* (1, 1, \ldots, 1, 0) &(i \le n), \\*[\smallskipamount]
    {\rm (ii)} & \col_j(X^{(k)}) \preceq^* (1, 1, \ldots, 1, 0) &(j \le n).
    \end{array} 
\end{equation}
If these conditions hold, there is a unique ASHM-completion at layer $k$ which is $A=[A_1,\ldots, A_{k-1}, Y, A_{k+1},\ldots, A_n]$. 
\end{theorem}
\begin{proof}
  Let $X^{(k)}=[x^{(k)}_{ij}]$. Assume first that $A=[A_1,\ldots, A_{k-1}, Z, A_{k+1},\ldots, A_n]$ is an ASHM-completion of layer $k$ of $A_{(k)}$. Then every line sum of $A$ is 1, so the sum of all these matrices (the $n$ layers) is $J_n$, and therefore $Z=J_n-X^{(k)}=Y$. Thus, such a completion is unique, if it exists. Also, condition $(\ref{eq:completion-cond1})$ holds as every line is alternating.
Since $Z=Y=[y_{ij}]$ is an ASM, for every $p< n$, the sum of its first $p$ entries in a row is nonnegative, so for each $i \le n$, 
\[ 
   (*)\;\;\; 0 \le \sum_{j=1}^p y_{ij}=  \sum_{j=1}^p (1-x^{(k)}_{ij})=p-\sum_{j=1}^p x^{(k)}_{ij},
\]
that is,  $\sum_{j=1}^p x^{(k)}_{ij} \le p$. Moreover, 
\[
 \sum_{j=1}^n x^{(k)}_{ij} =\sum_{j=1}^n \sum_{s\not = k}  a_{ijs}=\sum_{s\not = k} \sum_{j=1}^n a_{ijs}=(n-1)\cdot 1=n-1.
 \]
 This proves that $\row_i(X^{(k)}) \preceq^* (1, 1, \ldots, 1, 0)$ for each $i \le n$. Similarly, one obtains 
$\col_j(X^{(k)}) \preceq^* (1, 1, \ldots, 1, 0)$ for each $j \le n$. This shows the necessity of the condition.

Conversely, assume that conditions (\ref{eq:completion-cond1}) and (\ref{eq:completion-cond2}) hold. Define 
\[
  A=[A_1,\ldots, A_{k-1}, Y, A_{k+1},\ldots, A_n] = [a_{ijs}].
\]
We verify that $A$ is an ASHM. Since $Y=J_n-X^{(k)}$ and $X^{(k)}=\sum_{s\not = k} A_s$, $\sum_s a_{ijs}=1$ for each $i,j \le n$. Let $i,j \le n$. Consider the  line $(a_{ij1}, a_{ij2}, \ldots, a_{ijn})$ in $A$, and  its subvectors 
$u=(a_{ij1}, a_{ij2}, \ldots, a_{ij,k-1})$ and $v=(a_{ij,k+1}, a_{ij,k+2}, \ldots, a_{ijn})$. By  (\ref{eq:completion-cond1}) $u$ is $(1,*)$-alternating and $v$ is $(*,1)$-alternating. Define  $\sigma_1=\sum_{s=1}^{k-1} a_{ijs}$ and $\sigma_2=\sum_{s=k+1}^{n} a_{ijs}$. Then we have the following cases: 

\noindent (i) $\sigma_1=\sigma_2=0$. Then the last nonzero of $u$ (if any)  is $-1$,  and the first nonzero of $u$ (if any)  is $-1$. Moreover, $a_{ijk}=y_{ij}=1$, so the line is alternating.

\noindent (ii) $\sigma_1=1$, $\sigma_2=0$. Then the last nonzero of $u$ is 1, and the first nonzero of $v$ (if any) is $-1$, and $a_{ijk}=y_{ij}=1-(0+1)=0$, so the line is alternating.

\noindent (iii) $\sigma_1=0$, $\sigma_2=1$. Similar to case (ii).

\noindent (iv) $\sigma_1=1$, $\sigma_2=1$. Then the last nonzero of $u$ (if any), and the first nonzero of $v$,  is 1, and $a_{ijk}=y_{ij}=1-(1+1)=-1$, so the line is alternating.

It only remains to show that $Y$ is an ASM. Let $i \le n$. The computation in $(*)$ shows that, for each $p<n$,  $\sum_{j=1}^p y_{ij} \ge 0$ as $\sum_{j=1}^p x^{(k)}_{ij} \le p$ by  (\ref{eq:completion-cond1}). Moreover, as each $A_s$ is an ASM,  
\[
  \sum_{j=1}^n y_{ij}=n-\sum_{j=1}^n x^{(k)}_{ij}= n-\sum_{j=1}^n \sum_{s\not = k} a_{ijs}=
  n- \sum_{s\not = k} \sum_{j=1}^n a_{ijs} =n-(n-1)\cdot 1 = 1.
\]  
Similarly, we obtain, for each $j \le n$, that $\sum_{i=1}^p y_{ij} \ge 0$ for each $p <n$, and $\sum_{i=1}^n y_{ij} =1$. This shows that $Y$ is an ASM, and therefore $A$ is an ASHM.
\end{proof}

In what follows we shall make use of  the following classical decomposition  result (\cite{Brualdi06,BR91}).

\begin{theorem}
 \label{thm:nonneg-decomp}
    Let $A$ be an $n \times n$ nonnegative integral matrix with equal row and column sums, say equal to $k$. Then $A$ may be decomposed as the sum of $k$ permutation matrices.
\end{theorem}

Let $k<n$ and let $A=[A_1, A_2, \ldots, A_{k}]$ be an $n \times n \times k$ hypermatrix where $A_1, A_2, \ldots, A_k$ are  ASMs. We consider the problem of extending $A$ by inserting $(n-k)$\ $n \times n$ ASMs $A_{k+1}, \ldots, A_n$ so that $A'=[A_1,\ldots, A_{k}, A_{k+1}, \ldots, A_n]$ is an $n \times n \times n$ ASHM. An obvious necessary condition on the vertical lines is the following:
\begin{equation}
   \label{eq:completion-cond3}
    \begin{array}{llll}
   (a_{ij1}, \ldots, a_{ijk}) &\mbox{\rm is $(1,*)$-alternating} &(1 \le i,j \le n). \\*[\smallskipamount]
    \end{array} 
\end{equation}

The next theorem says that no other condition is  needed. 

\begin{theorem}
 \label{thm:SASM-completion2}
  Let $k<n$, and let $A_1, A_2, \ldots, A_{k}$ be   $k$  $n\times n$ ASMs satisfying  $(\ref{eq:completion-cond3})$. Then there exist $n \times n$ ASMs $A_{k+1},  \ldots, A_n$ such that $A'=[A_1,\ldots, A_{k}, A_{k+1}, \ldots, A_n]$ is an $n \times n \times n$ ASHM. 
  Moreover, each of these additional ASMs  $A_j$ $(k<j\le n)$ may be chosen as a permutation matrix. 
\end{theorem}
\begin{proof}
 Let $S=\sum_{s=1}^k A_s$. Then each line sum in $S$ is $k$. Moreover, it follows from condition $(\ref{eq:completion-cond3})$ that $S$ is a $(0,1)$-matrix. Let $T=J_n-S$ which is also a $(0,1)$-matrix with each line sum $n-k$. Clearly $S$ and $T$ have disjoint supports. 
 
Note that if $S$ has a 1 in position $(i,j)$ it means that $\sum_{s=1}^k a_{ijs}=1$, and the vertical line $(a_{ij1}, \ldots, a_{ijk})$ is alternating. Since all line sums in $T$ are $n-k$, by Theorem \ref{thm:nonneg-decomp},  there are permutation matrices $A_{k+1}, \ldots, A_n$ such that 
\[
    T=\sum_{s=k+1}^n A_s.
\]
Define the $n \times n \times n$ hypermatrix  $A'=[A_1,\ldots, A_{k}, A_{k+1}, \ldots, A_n]=[a_{ijs}]$. Then $A'$ is an ASHM because: 

\noindent (i) In positions $(i,j)$ where $S$ has a 1 the line $(a_{ij1}, \ldots, a_{ijk})$ is alternating and has sum 1 while $a_{ijs}=0$ for $s>k$.  

\noindent (ii) In positions $(i,j)$ where $S$ has a 0, $T$ has a 1, so exactly one $A_s$, for $s>k$, has a 1 in that position $(i,j)$. Moreover, the line $(a_{ij1}, \ldots, a_{ijk})$ is $(1,*)$-alternating so its last nonzero, if any, is $-1$. 
So the whole line $(a_{ij1}, \ldots, a_{ijn})$ is alternating.
\end{proof}

We remark that one may find the permutation matrices $A_{k+1}, \ldots, A_n$ in Theorem \ref{thm:SASM-completion2} efficiently (in polynomial time)  since matching algorithms may be be used to find the decomposition in Theorem \ref{thm:nonneg-decomp}.

\begin{example}
\label{exa:completion-2layers}
{\rm  Consider Theorem \ref{thm:SASM-completion2}, and its proof,  with  $k=2$ and 
\[
A_1=\left[\begin{array}{r|r|r|r}
&1&& \\ \hline
&&1& \\ \hline
1&&& \\ \hline
&&& 1\end{array}\right], \;\; 
A_2=\left[\begin{array}{r|r|r|r}
&&1& \\ \hline
1&&-1&1 \\ \hline
&1&& \\ \hline
&&1& \end{array}\right].
\]
Then 
\[
S=\left[\begin{array}{r|r|r|r}
&1&1& \\ \hline
1&&&1 \\ \hline
1&1&& \\ \hline
&&1&1 \end{array}\right],  \;\; 
T=\left[\begin{array}{r|r|r|r}
1&&&1 \\ \hline
&1&1& \\ \hline
&&1&1 \\ \hline
1&1&& \end{array}\right]. 
\]
So,  we can extend $[A_1, \; A_2]$ into an ASHM by adding the permutation matrices
\[
A_3=\left[\begin{array}{r|r|r|r}
1&&& \\ \hline
&&1& \\ \hline
&&&1 \\ \hline
&1&& \end{array}\right], \;\;
A_4=\left[\begin{array}{r|r|r|r}
&&&1 \\ \hline
&1&& \\ \hline
&&1& \\ \hline
1&&& \end{array}\right].
\]
}
\end{example} \endproof

Recall that $\sigma_-(A)$ denotes the number of negative entries of a matrix, or hypermatrix,  $A$.
By Theorem \ref{th:diamond}  the maximum number $m_n$ of nonzeros in an $n\times n\times n$  ASHM is given by $m_n=\frac{n(n^2+2)}{3}$. Therefore the maximum number $m^-_n$ of $-1$'s in an ASHM is 
\begin{equation}\label{eq:mn-}
   m^-_n=\sigma_-({\mathfrak D}_n)=m_n-n^2=\frac{n(n^2-3n+2)}{3}.
\end{equation}

Let $n\ge 3$. Our goal now is to show that, for any nonnegative integer $t \le m^-_n$, there exists an ASHM $A$ with $\sigma_-(A)=t$ (for $n\le 2$ there is nothing to show). 

First, we describe a class of ASMs that are obtained from the diamond ASMs. The construction is illustrated in Example \ref{ex:neg-ext1}. 
Let $1<k< n$ and consider $F^k_n$. For $1 \le j \le k$, the {\em $j'$th positive diagonal} of $F^k_n$ consists  of the positions $(j,k-j+1), (j+1,k-j+2), \ldots, (n+j-k,n-j+1)$. Similarly, for $j \le k-1$  the {\em $j'$th negative diagonal} of $F^k_n$ consists  of the positions $(j+1,k-j+1), (j+2,k-j+2), \ldots, (n+j-k,n-j+2)$. 
 Define $p=n-k$ which is the number of $-1$'s in each of the negative diagonals of $F^k_n$. We have $\sigma_-(F^k_n)=(k-1)(n-k)$. Let $r$ and $s$ be integers with $0 \le r <k$ and $0 \le s < n-k$.
Let $F^{k,r,s}_n=[a_{ij}]$ be the $n \times n$ matrix obtained from $F^k_n$ as follows:

(i) If $r>0$, let each entry in the first $r$ positive diagonals and the first $r$ negative diagonals be zero (some of these entries may be changed again in step (ii));

(ii) If $p$ is even, let $a_{1n}=a_{2,n-1}=\cdots = a_{r,n-r+1}=1$. Otherwise, when $p$ is odd, there are two subcases: (a) $r$ is odd; then let the entry be 1 in each of the $r$ first positions in the sequence $(1,n)$, $(2,n-2)$, $(3,n-1)$, $(4,n-4)$, $(5,n-3)$, $\ldots$; (b) $r$ is even; let the entry be 1 in each of the $r$ first positions in the sequence $(1,n-1)$, $(2,n)$, $(3,n-3)$, $(4,n-2)$,  $\ldots$.

(iii) If $s>0$, let the last $s$ entries in the last negative diagonal be zero, and the last $s+1$ entries of the last positive diagonal be zero. Finally, let $a_{k+s,1}=1$.

Define 
\[
   K_{n,k}=\{(i,i+k-1): 1 \le i \le n-k+1\} \cup \{(j+k-1,j): 1 \le j \le n-k+1\}
\]
which consists of positions in the upper right triangle and the lower left triangle.

\begin{lemma}
 \label{lem:Fkrs}
   Each matrix $F^{k,r,s}_n$ is an ASM, and $\sigma_-(F^{k,r,s}_n)=(k-r-1)(n-k)-s$. $F^{k,r,s}_n$ is obtained from $F^k_n$ by replacing some zeros in $K_{n,k}$ by $1$, and replacing some nonzeros outside $K_{n,k}$ by zero. 
\end{lemma}
\begin{proof}
This follows from  the construction of $F^{k,r,s}_n$.
\end{proof}

\begin{example}
\label{ex:neg-ext1}
{\rm
Let $n=8$ and $k=5$, so 
\[
F^5_8=\left[
\begin{array}{r|r|r|r|r|r|r|r}
&&&&1&&&\\ \hline
&&&1&-1&1&&\\ \hline
&&1&-1&1&-1&1&\\ \hline
&1&-1&1&-1&1&-1&1\\ \hline
1&-1&1&-1&1&-1&1&\\ \hline
&1&-1&1&-1&1&&\\ \hline
&&1&-1&1&&&\\ \hline
&&&1&&&&
\end{array}
\right].
\]
Then  $p=n-k=3$ is odd. For $r=2$ and $s=0$,  we obtain 
\[
F^{5,2,0}_8=\left[
\begin{array}{r|r|r|r|r|r|r|r}
&&&&&&1&\\ \hline
&&&&&&&1\\ \hline
&&1&&&&&\\ \hline
&1&-1&1&&&&\\ \hline
1&-1&1&-1&1&&&\\ \hline
&1&-1&1&-1&1&&\\ \hline
&&1&-1&1&&&\\ \hline
&&&1&&&&
\end{array}
\right].
\]
Here the first two positive and negative diagonals of $F^k_n$ now contain zeros, and, as a compensation, we have the two ones in the upper right corner. 
For $r=1$ and $s=2$, the first positive and negative diagonals are replaced by zeros, and a modification is done in the lower left corner: 
\[
F^{5,1,2}_8=\left[
\begin{array}{r|r|r|r|r|r|r|r}
&&&&&&&1\\ \hline
&&&1&&&&\\ \hline
&&1&-1&1&&&\\ \hline
&1&-1&1&-1&1&&\\ \hline
&&1&-1&1&-1&1&\\ \hline
&&&1&-1&1&&\\ \hline
1&&&-1&1&&&\\ \hline
&&&1&&&&
\end{array}
\right].
\]

}\end{example}  \endproof

\begin{theorem}
 \label{thm:all-integers}
  Let $n$ be a positive integer. Then, where $m_n^-$ is given by $(\ref{eq:mn-})$, for every integer $0 \le t \le m^-_n$, there exists an ASHM $A$ with $\sigma_-(A)=t$. 
\end{theorem}
\begin{proof}
 First, note that any permutation ASHM does not have any negative entries. Let ${\mathfrak D}_n=[A_1,A_2,\ldots, A_n]$ be the diamond ASHM of size $n \times n \times n$. Then  ${\mathfrak D}_n$ has $m^-_n$ negative entries. Let  $1 < t <m^-_n$. Choose $k$ minimal such that the hypermatrix $[A_1, A_2,\ldots, A_k]$ has at least $t$ negative entries. So $k \le n-1$. 

 We claim that there exists an $n \times n$ ASM $B$ such that each vertical line in the $n \times n \times k$ hypermatrix $A'=[A_1,A_2,\ldots,A_{k-1},B]$ is $(1,*)$-alternating, i.e., $(\ref{eq:completion-cond3})$ holds and, moreover, $A'$ has $t$ negative entries. 

The claim follows from Lemma \ref{lem:Fkrs} as we can choose $B=F^{k,r,s}_n$ for suitable $r$ and $s$ so that $B$ has $t$ negative entries (the $r$ and $s$ are unique). Then $B$ is an ASM and the $(1,*)$-alternating property follows from the second part of the lemma, as all the matrices $A_1,  A_2, \ldots, A_{k-1}$ have only zeros in positions in $K_{n,k}$. 

Now, due to the claim, we apply Theorem \ref{thm:SASM-completion2}  to obtain $n \times n$ permutation matrices $P_s$ ($k+1\le s\le n$) such that $A=[A_1,A_2,\ldots,A_{k-1},B, P_{k+1}, P_{k+2}, \ldots, P_n]$ is an ASHM, and has exactly $t$ negative entries as desired.
\end{proof}

We remark that the ASHM $A$ constructed in the proof of Theorem \ref{thm:all-integers} is extreme in that it has the maximum number of elements in the first $k-1$ horizontal layers, and no negative entries in the last $n-k$ layers.

We now show that there exists an $n\times n\times n$ ASHM whose $-1$'s
are confined to the $k$th horizontal level and whose  number  is any number below the maximum possible at the $k$th horizontal level.

\begin{theorem}
 \label{thm:subset-perm-ext}
  Let $0 \le t \le \sigma_-(F^k_n)$. Then there exists an $n \times n$ ASM $A_k$ with $t$ negative entries and  $n\times n$ permutation matrices $P_1,\ldots,P_{k-1},P_{k+1},\ldots,P_n$ such that
\[
  A=[P_1,\ldots,P_{k-1},A_k,P_{k+1},\ldots,P_n]
\]
is an $n\times n\times n$ ASHM.
\end{theorem}
\begin{proof}
 We use the construction in the proof of Theorem \ref{thm:all-integers}. This shows that there is an ASM $A_k$ with $t$ negative entries and permutation matrices $P_{k+1}, \ldots, P_n$ such that 
 $A'=[F^1_n,\ldots,F^{k-1}_n,A_k,P_{k+1},\ldots,P_n]$ is an ASHM. 
 
 Consider the subhypermatrix $A''=[A_k,P_{k+1},\ldots,P_n]$ obtained from $A'$ by deleting the first $k-1$ horizontal planes. Then, by reversing the order of these planes, each vertical line is $(1,*)$-alternating, so we apply Theorem \ref{thm:SASM-completion2}, and there exist $n \times n$ permutation matrices $P_s$ ($1\le s\le k-1$) such that 
 \[
  A=[P_1,P_2,\ldots, P_{k-1},A_k,  P_{k+1}, P_{k+2}, \ldots, P_n]
\]
is an ASHM, as desired. 
\end{proof}

Recall that a matrix is {\it convex} provided the nonzeros  in each of its lines  are consecutive.
Examples of convex ASMs are the $n\times n$ ASMs $F_n^k$, the $k$th horizontal ASM of the diamond ASHM ${\mathfrak D}_n$. An $n\times n\times n$ hypermatrix is {\it convex} provided the nonzeros  in each of its three types of lines are consecutive. The diamond ASHM ${\mathfrak D}_n$ is convex.  We call an ASM {\em minus-convex} provided in every row and column all entries between two negative entries (if any) are nonzero, i.e., each of its rows and columns has the form
\[
   0, \ldots, 0, 1, 0, \ldots, 0, -1, 1, -1, 1, \ldots, -1, 0, \ldots, 0, 1, 0, \ldots, 0
\]
where the subsequences of zeros can be void.

\begin{theorem}\label{th:convex}
Let $B$ be an $n\times n$ matrix obtained from an $n\times n$ ASM $A$ 
by replacing some nonzeros with zeros in such a way that the result is convex.  Then $B$ can be completed to an $n\times n$ ASM $B'$ by changing some zeros to ones.
\end{theorem}

Before giving a proof of Theorem \ref{th:convex}, which provides a simple algorithm to obtain $B'$, we illustrate the theorem with an example.

\begin{example}
\label{ex:convex}
{\rm
 Consider the ASM
\[A=F_{10}^{5}=\left[\begin{array}{r|r|r|r|r|r|r|r|r|r}
&&&&\cellcolor{blue!25}1&&&&&\\ \hline
&&&\cellcolor{blue!25}1&\cellcolor{blue!25}-1&\cellcolor{blue!25}1&&&&\\ \hline
&&\cellcolor{blue!25}1&\cellcolor{blue!25}-1&1&-1&\cellcolor{blue!25}1&&&\\ \hline
&\cellcolor{blue!25}1&-1&1&-1&1&\cellcolor{blue!25}-1&\cellcolor{blue!25}1&&\\ \hline
\cellcolor{blue!25}1&\cellcolor{blue!25}-1&1&-1&1&-1&1&-1&\cellcolor{blue!25}1&\\ \hline
&1&-1&1&-1&1&-1&1&-1&\cellcolor{blue!25}1\\ \hline
&&1&-1&1&-1&1&-1&1&\\ \hline
&&&1&-1&1&-1&\cellcolor{blue!25}1&&\\ \hline
&&&&ç1&\cellcolor{blue!25}-1&\cellcolor{blue!25}1&&&\\ \hline
&&&&&\cellcolor{blue!25}1&&&&\end{array}\right]\]
where the shaded entries are to be replaced with zeros to obtain an $10\times 10$ convex matrix $B$. Then a completion of $B$ to an ASM obtained by changing certain zeros to ones is
\[B'=\left[\begin{array}{r|r|r|r|r|r|r|r|r|r}
&&&&&1&&&&\\ \hline
&&1&&&&&&&\\ \hline
&&&&1&-1&&1&&\\ \hline
1&&-1&1&-1&1&&&&\\ \hline
&&1&-1&1&-1&1&-1&1&\\ \hline
&1&-1&1&-1&1&-1&1&-1&1\\ \hline
&&1&-1&1&-1&1&-1&1&\\ \hline
&&&1&-1&1&-1&1&&\\ \hline
&&&&1&&&&&\\ \hline
&&&&&&1&&&\end{array}\right].\]}
\end{example}

\begin{proof} (of Theorem \ref{th:convex})
If $B$ does not have any $-1$'s (in particular, if $A$ does not have any $-1$'s), then $B$ consists of $p\ge 0$\ 1's no two on the same line, and $B$ can be completed to a permutation matrix and hence an ASM. We now assume that $B$ contains at least one $-1$. To proceed we consider the following transformation of an ASM.

Let $A=[a_{ij}]$  be an ASM with at least one negative entry. Choose $(k,l)$ with $k+l$ minimal such that $a_{kl}=-1$. Then there is  a unique $l'<l$ with $a_{kl'}$ nonzero, and we have $a_{kl'}=1$. Similarly, there is  a unique $k'<k$ with $a_{k'l}$ nonzero and $a_{k'l}=1$. By the choice of $k, l$ there is no negative entry in the leading $k \times l$ submatrix, apart from $a_{kl}$. Moreover, $a_{k'l'}=0$. Let $C=[c_{ij}]$ be obtained from $A$ by letting $c_{kl}=c_{k'l}=c_{kl'}=0$,  $c_{k'l'}=1$, and, otherwise, $c_{ij}=a_{ij}$. Then $C$ is an ASM and $\sigma_-(C)=\sigma_-(A)-1$. This transformation makes a change in the upper-left corner of the matrix $A$. Clearly a similar transformation may be performed in the upper-right, lower-left or lower-right corner of $A$. Each of these will be called a {\em corner transform} of an ASM. 

Note that if $A$ is a minus-convex ASM and $C$ is obtained from $A$ by a corner transform, then $C$ is also minus-convex. 

The theorem follows by a sequence of corner transforms:  Let $A$ and $B$ be as stated in the theorem. If $B$ has a $-1$ wherever $A$ has a $-1$, then we let $B'=A$. Now assume that there is at least one position in which $A$ has a $-1$ and $B$ has a $0$.  Choosing $(k,l)$ as above, we use a corner transformation and obtain a minus-convex ASM which has a $-1$ in every position that $B$ has a $-1$. Proceeding inductively  we obtain the desired ASM $B'$.
\end{proof}

\begin{example}\label{ex:convex2}{\rm 
Let
\[A=F_5^3=\left[\begin{array}{r|r|r|r|r}
&&\cellcolor{blue!25}1&&\\ \hline
&\cellcolor{blue!25}1&\cellcolor{blue!25}-1&1&\\ \hline
\cellcolor{blue!25}1&\cellcolor{blue!25}-1&1&-1&1\\ \hline
&\cellcolor{blue!25}1&\cellcolor{blue!25}-1&1&\\ \hline
&&\cellcolor{blue!25}1&&\end{array}\right]\]
where the shaded entries are to be replaced with zeros to obtain a convex matrix $B$. Then using corner transformations we obtain an ASM:
\[A\rightarrow \left[\begin{array}{r|r|r|r|r}
&1&&&\\ \hline
&&&1&\\ \hline
1&-1&1&-1&1\\ \hline
&1&-1&1&\\ \hline
&&1&&\end{array}\right]\rightarrow
\left[\begin{array}{r|r|r|r|r}
1&&&&\\ \hline
&&&1&\\ \hline
&&1&-1&1\\ \hline
&1&-1&1&\\ \hline
&&1&&\end{array}\right]\rightarrow
\left[\begin{array}{r|r|r|r|r}
1&&&&\\ \hline
&&&1&\\ \hline
&&1&-1&1\\ \hline
&&&1&\\ \hline
&1&&&\end{array}\right]=B'.
\]
}
\end{example} \endproof

\medskip
We now study another completion problem, where an ASM is given and we want to extend it to an ASHM in some way.
First, we introduce a generalization of an ASM. Consider a  $(0, \pm 1)$-matrix with all line sums equal to  1; such a matrix will be called a {\em semi-ASM}. Thus, if a line contains $p$ $-1$'s, then it has $(p+1)$\  $+1$'s. If $A$ is a PASHM, then 
its row-vertical-planes and column-vertical-planes are semi-ASMs.

The semi-ASMs have the following simple property. With a semi-ASM $A=[a_{ij}]$ of order $n$ we associate the bipartite graph $G_0(A)$ of $A$ with $n$ vertices in each color class and having an edge $ij$ whenever $a_{ij}=0$; thus the edges correspond to the zeros and not, as is often the case, to the nonzeros. If $G_0(A)$ has a cycle (which must be even), then by  putting $+1$ and $-1$ in the positions of $A$ of this cycle (alternating), we obtain a new matrix $B$ from $A$. Then $B$ is also a semi-ASM and has more nonzeros than $A$.

If $A$ and $B$ are semi-ASMs of order $n$ and  $B$ agrees in every nonzero entry of $A$, we say that $B$ is a {\em semi-ASM extension} of $A$. Thus, the cycle construction above produces an semi-ASM extension $B$ of $A$, and we call this a {\em cycle-extension}.

\begin{theorem}
 \label{thm:semi-ASM-extension}
 Let  $A$ be a semi-ASM of order $n$ where $n$ is odd. Then $A$ has a semi-ASM extension $B$ with all entries nonzero.
\end{theorem}
\begin{proof} The proof of this theorem uses a standard kind of argument.
 Assume $A$ has at least one zero entry (otherwise we are done). Let $L$ denote the set of lines in $A$ that contain at least one zero entry. Since $A$ is a semi-ASM, each line has an odd number of nonzero entries. Therefore, as $n$ is odd, each line in $L$ has a positive even number  of zeros. Consider the subgraph of $G_0(A)$ induced by the vertices corresponding to $L$ (so we have just removed isolated vertices, that is, vertices corresponding to lines of all 1's). In this subgraph each vertex has an even, nonzero degree, and therefore the subgraph  contains a cycle. Now, use this cycle and perform a cycle-extension of $A$. This gives a new matrix $A'$ with fewer zeros than $A$, and $A'$ is a semi-ASM. We may repeat this process, and get a sequence of semi-ASMs with fewer zeros, and eventually we obtain a semi-ASM $B$ with no zeros, as desired.
\end{proof}
 
The next example shows that the property in Theorem  \ref{thm:semi-ASM-extension} does not hold in general when $n$ is even.

\begin{example}
 \label{ex:even-extension}
{\rm
Let $n=4$ and let the semi-ASM $A$ be given by 
\[
A=
\left[
 \begin{array}{rrrr}
    0 & 0 & 0 &1 \\
    1 & -1 & 1 & 0 \\
    1 & 1 & -1 & 0 \\
    -1 & 1 & 1 & 0 \\
 \end{array} 
 \right].
\]
To find an extension, in the first row, we need to change some 0 to $-1$ and another 0 to $+1$, but then we violate that all column sums should be 1. So a semi-ASM extension  does not exist.
\endproof
}
\end{example}

We say that two $n\times n$ $(0,1)$-matrices $A_1$ and $A_2$  are {\em disjoint} provided that their supports are disjoint, that is, provided that $A_1+A_2$ is also a $(0,1)$-matrix.

\begin{corollary}
 \label{cor:odd-ext-perm}
 Let $n=2k+1$ be odd, and let $A$ be a semi-ASM of order $n$. Then there is a decomposition of $J_n$ into pairwise disjoint permutation matrices $P_1, P_2, \ldots, P_n$ such that the ones in $P_i$ $(i \le k)$ cover all the $-1$'s in $A$, and the ones in $P_i$ $(k < i \le n)$  cover all the $1$'s in $A$. 
\end{corollary}
\begin{proof} 
 Apply Theorem \ref{thm:semi-ASM-extension} to $A$, and let $B$ be an extension with all entries nonzero. The $(0,1)$-matrix $B_1$ obtained from $B$ by replacing each $1$ by $0$, and multiplying the matrix by $-1$ has all line sums equal to $k$, and can therefore be written as the sum of $k$ pairwise disjoint permutation matrices. Similarly, the $(0,1)$-matrix $B_2$ obtained from $B$ by replacing each $-1$ by $0$ has all line sums equal to $k+1$, and can therefore be written as the sum of $k+1$ pairwise disjoint permutation matrices. These $n$ permutation matrices are pairwise disjoint, so the result follows.  
\end{proof}

We now use the previous results to solve the  completion problem mentioned above, where we want to extend a given ASM into an ASHM.

\begin{theorem}
 \label{thm:central-ASM-extension}
 Let  $B$ be an ASM of order $n$ where $n=2k+1$ is odd. Then there are $n-1$ permutation matrices $P_i$ $(i \le n, i \not =k)$ such that 
 \[
  A=
[P_1,\ldots, P_k, B, P_{k+1}, \ldots, P_n]
\]
 is an ASHM.
\end{theorem}
\begin{proof}
 First, by   Corollary \ref{cor:odd-ext-perm}  there exist pairwise disjoint permutation matrices $P_1, P_2, \ldots, P_k$ whose ones cover all the $-1$'s in $B$, but do not cover any of the $1$'s in $B$. Therefore the matrix 
 \[
    C=P_1+ \cdots + P_k+B
\]
is a $(0,1)$-matrix. Consider the $n \times n \times (k+1)$ hypermatrix $A'=[a_{ijs}]=[P_1,  \ldots,  P_k, B]$.  As $C$ is a $(0,1)$ matrix and the $P_i$'s are pairwise disjoint, each of the vertical  lines $(a_{ij1}, \ldots, a_{ij,k+1})$ is alternating ($i,j \le n$). So,  (\ref{eq:completion-cond3}) holds with $k$ replaced by $k+1$, and by   Theorem \ref{thm:SASM-completion2}  there are permutation matrices $P_{k+1}, P_{k+2}, \ldots, P_n$
such that 
$A=[P_1,\ldots, P_k, B, P_{k+1}, \ldots, P_n]$ is an ASHM.
\end{proof}

\begin{example}{\rm 
Let $n=3$ and 
\[
B=\left[\begin{array}{rrr}
0&1&0\\
1&-1&1\\
0&1&0\end{array}\right]
\]
Then, in Theorem \ref{thm:central-ASM-extension}, we may let $P_1$ be the identity matrix (it covers the $-1$, but none of the ones in $B$), and this gives the following ASHM with $B$ as the middle plane:
\[
A:=\left[\begin{array}{ccc}
1&0&0\\
0&1&0\\
0&0&1\end{array}\right]\nearrow
\left[\begin{array}{rrr}
0&1&0\\
1&-1&1\\
0&1&0\end{array}\right]\nearrow
\left[\begin{array}{ccc}
0&0&1\\
0&1&0\\
1&0&0\end{array}\right].
\]  \endproof
}
\end{example}

We now consider the case of even $n$. First we make a general definition. Let $B_1$ and $B_2$ be $n\times n$ $(0,1,-1)$-matrices.
We define $B_1$ and $B_2$ to be {\it sign-disjoint} provided $B_1+B_2$ is also a $(0,1,-1)$-matrix. Thus $B_1$ and $B_2$ are sign-disjoint if and only if they have neither 1's  in the same position nor $-1$'s in the same position. If $B_1$ and $B_2$ are $(0,1)$-matrices, then $B_1$ and $B_2$ are sign-disjoint if and only if they are disjoint as previously defined.  We also observe that if $A=[A_1,A_2,\ldots,A_n]$ is an $n\times n\times n$  ASHM, then $A_i$ and $A_{i+1}$ are sign-disjoint for each $i=1,2,\ldots,n-1$.

\begin{lemma}\label{lem:evendisjoint} Let $n$ be an even integer, and
let $B_1$ and $B_2$ be $n\times n$ sign-disjoint ASMs. Then the line sums of $B_1+B_2$ all equal $2$, and the  number of zeros in each row and column of  $B_1+B_2$ is even.
\end{lemma}

\begin{proof} Since $B_1$ and $B_2$ are ASMs, all of their line sums equal 1, and hence all of the line sums of $B_1+B_2$ equal 2.
Also, since $B_1$ and $B_2$ are ASMs, each contains an odd number of nonzeros in each row and column. 
Consider some  row (or column) $i$ of $B_1$ and $B_2$. Let the number of nonzero positions in row $i$ of $B_1$
(respectively, $B_2$) be $n_1$ (respectively, $n_2$), and let $c$ be  the number of  positions in row $i$ in which $B_1$ has a $1$ and $B_2$ has a $-1$ or the other way around. Then the number of nonzeros in row $i$ of $B_1+B_2$ equals $n_1+n_2-2c$, an even number. Since $n$ is also even, the number of zeros in row $i$ of $B_1+B_2$ is even.\end{proof}

\begin{corollary}\label{cor:semi-ASM-extension-even}
Let $n$ be an even integer, and
let $B_1$ and $B_2$ be $n\times n$ sign-disjoint ASMs. Then there is a  $(0,1,-1)$-matrix $B$ which extends $B_1+B_2$ with all entries nonzero and all line sums equal to $2$. Moreover, there are pairwise disjoint  permutation matrices $P_1,P_2,\ldots,P_n$ such that the permutation matrices $P_1,\ldots,P_{\frac{n}{2}-1}$ cover all the $-1$'s of $B_1+B_2$, and the permutation matrices $P_{\frac{n}{2}},\ldots,P_n$ cover all the $1$'s of $B_1+B_2$.
\end{corollary}

\begin{proof} The proof follows as in  the proofs of Theorem
\ref{thm:semi-ASM-extension} and  Corollary \ref{cor:odd-ext-perm}.\end{proof}

\begin{theorem}
 \label{thm:central-ASM-extension-even} 
Let $B_1$ and $B_2$ be $n\times n$ sign-disjoint ASMs where $n=2k$ is even. Then there are $n-2$ permutation matrices $P_i$ $(1\le i\le n, i\ne \frac{n}{2},\frac{n}{2}+1)$ such that
\[A=[P_1,\ldots,P_{\frac{n}{2}-1},B_1,B_2,P_{\frac{n}{2}+2},\ldots,P_n]\]
is an ASHM.
\end{theorem}

\begin{proof}
By Corollary \ref{cor:semi-ASM-extension-even}, there is a $(0,1,-1)$-matrix $B$ which extends $B_1+B_2$ with all entries nonzero and all 
line sums equal to $2$. Let $D$ be the $(0,1)$-matrix obtained from $B$ by replacing each $1$ with a 0 and multiplying the matrix by $-1$. Then all line sums of $D$ equal  $\frac{n}{2}-1$, and 
the matrix $P_1+\cdots+P_{\frac{n}{2}-1}$ is a $(0,1)$-matrix.
As in the proof of Theorem \ref{thm:central-ASM-extension} there are permutation matrices $P_{\frac{n}{2}+2},\ldots,P_n$ such that 
$A=[P_1,\ldots,P_{\frac{n}{2}-1},B_1,B_2,P_{\frac{n}{2}+2},\ldots,P_n]$ is an ASHM.
\end{proof}

\medskip

 Let $n\ge 2$ be an integer and let $A$ be an $n\times n$ ASM. An $n\times n$ ASM $B$ is an {\em ASM-mate} of $A$ provided that $A+B$ is a $(0,1,-1)$-matrix, that is, provided that $A$ and $B$ are sign-disjoint.

\begin{theorem}
 \label{thm:1-mate} 
Let $A$ be an ASM of order $n$. Then $A$ has an ASM-mate $B$ which may be  chosen as a permutation matrix.
\end{theorem}
\begin{proof}
 First, assume $n$ is odd, say $n=2k+1$. By Corollary \ref{cor:odd-ext-perm} there are pairwise disjoint permutation matrices $P_1, P_2, \ldots, P_n$ such that the ones in $P_i$ $(i \le k)$ cover all the $-1$'s in $A$, and the ones in $P_i$ $(k < i \le n)$  cover all the $1$'s in $A$. In particular,  $P_1$ covers only entries that are $-1$ or $0$ (by the pairwise disjointness, and as the ones are covered by other $P_i$'s). Therefore $A+P_1$ is a $(0,1,-1)$-matrix, and $P_1$ is an ASM-mate of $A$. 
 
 Next, assume $n$ is even, say $n=2k$, and let $A=[a_{ij}]$. Let $I$ be a nonempty subset of $\{1, 2, \ldots, n\}$, and let 
\[
   J_I=\{j \le n: a_{ij} \in \{0,-1\} \;\mbox{\rm for at least one $i \in I$}\}. 
 \]
 If $|I| \le k$, then $|J_I| \ge k \ge |I|$, since every row in $A$ has at most $k$ 1's (as $A$ is an ASM) and therefore at least $k$ entries in $\{0,-1\}$. Otherwise, $|I| > k$, and then every column in $A$ must contain a 0 or a $-1$ in some row in $I$, as every column of $A$ contains at most $k$ 1's. So, then $|J_I| =n \ge |I|$.
 
 This proves that $|J_I| \ge |I|$ for every nonempty subset $I$ of $\{1, 2, \ldots, n\}$, and by Hall's (marriage) theorem, there must be a permutation matrix $P$ whose ones are in the positions where $A$ has entries in $\{0,-1\}$. Therefore $A+P$ is a $(0,1,-1)$-matrix, and $P$ is an ASM-mate of $A$. 
\end{proof}



\medskip

By definition, if we delete planes in a hypermatrix, we obtain a subhypermatrix. We next show that any $(0, \pm 1)$-hypermatrix is the subhypermatrix of some ASHM. This result is useful for  constructing classes of (``random'')  ASHMs. A similar result for ASM submatrices was established in \cite{BKMS13}.

Let $A=[A_1, A_2, \ldots, A_p]=[a_{ijk}]$ be an $m \times n \times p$ hypermatrix where $A_1, A_2, \ldots, A_p$ are the horizontal-planes of $A$. Let $k<p$. Define the $(0,\pm 1)$-matrix   $C^{A,k} =[c_{ij}]$  of size $m \times n$ as follows. Let  $i \le m$ and $j \le n$. If the last nonzero, if any,  in 
the (partial) vertical line $(a_{ij1}, a_{ij2}, \ldots, a_{ij,k-1})$ is $-1$ (resp. $1$) and $a_{ijk}=-1$ (resp. $1$), we define $c_{ij}=1$ (resp. $c_{ij}=-1$).   All other entries of $C$ are zero. 
 The idea of the construction algorithm is to insert planes between $A_{k-1}$ and $A_k$ based on the matrix $C^{A,k}$.
The following result will be useful, see \cite{BR91}. A {\em subpermutation matrix} is a $(0,1)$-matrix with at most one 1 in every row and column.

\begin{theorem}
 \label{thm:decomp-subperm} \textup{(\cite{BR91})}
   Let $A$ be a $(0,1)$-matrix with maximum line sum $t$. Then there are $t$ pairwise disjoint subpermutation matrices $P_1, P_2, \ldots, P_t$ such that 
 \[
     A= P_1 + P_2 + \cdots + P_t.
 \]
\end{theorem}

We now formulate the announced result. 

\begin{theorem}
 \label{thm:1-mate-more} 
Let $A'$ be an $(0, \pm 1)$-hypermatrix of size $m' \times n' \times k'$. Then there is an integer $n\ge m',n',k'$ and   an $n\times n\times n$  ASHM $A$ with $A'$ as a subhypermatrix. In fact, $A$ may be obtained from $A'$ by repeatedly inserting new planes $($horizontal, column-vertical and row-vertical$)$ that are subpermutation matrices and  their negatives. 
\end{theorem}
\begin{proof}
 Let $A'=[A_1, A_2, \ldots, A_{k'}]=[a_{ijk}]$. We describe an algorithm for constructing the desired ASHM $A$. Start with the given 
 $A'$, and let $t_0$ be the maximum number of $-1$'s in a line of $A_1$ (so $t_0 \ge 0$). By Theorem \ref{thm:decomp-subperm} there are pairwise disjoint $m'\times n'$ subpermutation matrices $P^1_1, P^1_2, \ldots, P^1_{t_0}$ such that the union of their supports coincides with the set of positions of the $-1$'s in $A_1$. Let 
 \[
    A=[P^1_1, P^1_2, \ldots, P^1_{t_0}, A_1, A_2, \ldots,  A_{k'}]
 \]
 Therefore each  vertical line of the subhypermatrix 
 \[
    \bar{A}=[P^1_1, P^1_2, \ldots, P^1_{t_0}, A_1]
 \]
 of $A$
 is $(1,*)$-alternating (which includes the possibility of the zero vector). Next, if needed, we modify $A$ by adding horizontal planes between $A_1$ and $A_2$. Let $C=C^{A',t_0+2}$. Let $t_1$ (resp. $s_1$) be the maximum number of $1$'s (resp. $-1$'s) in a line of $C$. By Theorem \ref{thm:decomp-subperm}, there are pairwise disjoint subpermutation matrices $P^2_1, P^2_2, \ldots, P^2_{t_1}$ (resp. $Q^2_1, Q^2_2, \ldots, Q^2_{s_1}$) such that the union of their supports coincides with the set of positions of the $1$'s (resp. the $-1$'s)  in $C$, and, moreover,  $P^2_i$ and $Q^2_j$  are pairwise disjoint for each $i,j$. Then we update $A$ so that
 \[
    A=[P^1_1, P^1_2, \ldots, P^1_{t_0}, A_1, P^2_1, P^2_2, \ldots, P^2_{t_1}, -Q^2_1, -Q^2_2, \ldots, -Q^2_{s_1}, A_2, \ldots,  A_{k'}].
 \] 
  Then, by construction, each vertical  line in $A$ up to, and including, the matrix $A_2$  is $(1,*)$-alternating.  We continue in the same way, by inserting planes that are subpermutation matrices or their negatives between $A_i$ and $A_{i+1}$, so that the resulting matrix $A$ has $(1,*)$-alternating vertical lines up to its last plane $A_{k'}$. Then we identify those $(i,j)$ for which the vertical line $(i,j,*)$ misses a final 1, to make the line alternating, and add planes as above to achieve that. As a result, each vertical line in $A$ is alternating (and, therefore, nonzero). 
  
 The next step is to consider the row-vertical planes $A_{i**}$ in $A$.  We repeat the procedure above by inserting suitable row-vertical planes so that new hypermatrix $A$ has alternating row lines. Note that these plane insertions do not affect the alternating property of the vertical lines because every vertical line is included in (exactly) one row-vertical plane. 
 
 Finally, we repeat the procedure another time, now for the column-vertical planes $A_{*j*}$ in $A$. Again, the previous alternating properties are maintained. The resulting $m \times n \times k$ matrix $A$ has only alternating lines in each direction. Since every line sum in $A$ is 1,  summation shows that  $m=n=k$, so $A$ is an ASHM as desired.
 \end{proof}

Finally we mention that a different kind of ASM completion problem was considered in \cite{BK15}

\section{Coda}
\label{sec:coda}

As a generalization of alternating sign matrices and latin squares, we have initiated the study of the fascinating class of ASHMs and the associated class of ASHM latin squares. In view of the results obtained, there are many possible directions to pursue. In particular, a characterization of ASHM-LSs would be very interesting. A specific question is the following.

\begin{problem}\label{prob:maxnumber}{\rm 
What is the maximum number of times  an integer can occur as an entry in an $n\times n$ ASHM-LS? The $7\times 7$ ASHM-LS in (\ref{eq:n=7}) and the subsequent discussion shows that an integer can occur $2n$ times in an $n\times n$ ASHM-LS. Is the maximum equal to $2n$?
}
\end{problem}

As a  step in characterizing the entries of an ASHM-LS, we would like to characterize  the  weighted projections of ASMs.

\begin{conjecture}\label{conj:projection}{\rm
Let $(c_1,c_2,\ldots,c_n)$ be a vector of positive integers such that
$(c_1,c_2,\ldots,c_n)\preceq (n,n-1,\ldots,2,1)$. Then there exists an $n\times n$ ASM whose weighted vertical projection is $(c_1,c_2,\ldots,c_n)$.

}
\end{conjecture}

Another possible and worthwhile direction
is provided by
the concept of orthogonality of latin squares which is fundamentally important and well-studied  in combinatorics.

Let  
\begin{equation}\label{eq:latinsquares}
L_1=1P_1+2P_2+\cdots +nP_n \;\mbox{  and } \;L_2=1Q_1+2Q_2+\cdots+nQ_n
\end{equation}
be two $n\times n$ latin squares, where $P_1,P_2,\ldots,P_n$ and $Q_1,Q_2,\ldots,Q_n$ are $n\times n$ permutation matrices with $P_1+P_2+\cdots +P_n=J_n$
and $Q_1+Q_2+\cdots+Q_n=J_n$. Then $L_1$ and $L_2$ are {\it orthogonal} provided when they are juxtaposed to form a matrix $L_1*L_2$ of ordered pairs, all pairs $(i,j)$ with $1\le i,j\le n$ occur
(equivalently, there are no repeats); $L_1$ and $L_2$ are then  called  {\it orthogonal LS-mates}.
If $n>2$ but $n
\ne 6$, the celebrated theorem of Bose, Parker, and Shrikhande
(see \cite{R1963} for a discussion and more references) asserts that there exists a pair of orthogonal $n\times n$ latin squares. It is well-known that for each $n\ge 2$, the maximum number of pairwise orthogonal latin squares is bounded by $(n-1)$.

Another way to view orthogonality of latin squares is the following:
The $n\times n$ latin squares $L_1$ and $L_2$ in (\ref{eq:latinsquares}) are orthogonal if and only if  the standard inner products (sum of products of corresponding entries) of the $P_i$ and the $Q_j$ satisfy
\[\langle P_i, Q_j \rangle=1\quad  (1\le i,j\le n),\]
 equivalently,
each pair $P_i$ and $Q_j$ have exactly one 1 in common positions. 
Let us define  $P_i$ and $Q_j$ to be {\it orthogonal permutation-mates} when $\langle P_i, Q_j \rangle=1$. Thus the latin squares $L_1$ and $L_2$ in (\ref{eq:latinsquares})  are orthogonal LS-mates if and only if $P_i$  and $Q_j$ are orthogonal permutation-mates for all $i$ and $j$.
If $n\ge 2$, every $n\times n$ permutation matrix has an orthogonal permutation-mate (in fact, exactly $nD_{n-1}$ orthogonal mates where $D_{n-1}$ is the $(n-1)$st derangement number). But, as is well known,  not every latin square has an orthogonal LS-mate.

Now consider two ASHM-LS's.
\begin{equation}\label{eq:ashm-ls}
L(A)=1A_1+2A_2+\cdots+nA_n\mbox{ and }L(B)=1B_1+2B_2+\cdots+nB_n,
\end{equation}
 where $A=[A_1,A_2,\ldots,A_n]$ and 
$B=[B_1,B_2,\ldots,B_n]$ are $n\times n\times n$ ASHMs. Thus
\[J_n=A_1+A_2+\cdots+A_n\mbox{ and } J_n=B_1+B_2+\cdots+B_n.\]
Generalizing orthogonality of latin squares, with the $A_i$ and $B_j$ playing the role of the permutation matrices $P_i$ and $Q_j$ above,
we define the ASHM-LS's $L(A)$ and $L(B)$ to be {\it orthogonal} provided
the inner products of the $A_i$ with the $B_j$ satisfy
\[\langle A_i,  B_j \rangle=1\quad (1\le i,j\le n).\]
If $L(A)$ and $L(B)$ are orthogonal, then we also say that $L(A)$ and $L(B)$ are {\it orthogonal ASHM-LS mates}.  It follows from the theorem of Bose, Parker, and Shrikhande  that, if $n>2$ but $n\ne 6$, then there exists a pair of orthogonal ASHM-LS's.  Extending our definition above from permutation matrices to ASHMs, we define the ASMs
  $A_i$ and $B_j$ to be  {\it orthogonal ASM-mates} when $\langle A_i,  B_j \rangle=1$. 

   \begin{example}\label{ex:mates}{\rm   The
ASM
\[\left[\begin{array}{rrr}
0&1&0\\
1&-1&1\\
0&1&0\end{array}\right]\]
does not have an orthogonal ASM-mate. The ASMs
\[\left[\begin{array}{rrrrr}
0&0&1&0&0\\
0&1&-1&1&0\\
1&-1&1&-1&1\\
0&1&-1&1&0\\
0&0&1&0&0\end{array}\right],
\left[\begin{array}{ccccc}
0&1&0&0&0\\
1&0&0&0&0\\
0&0&1&0&0\\
0&0&0&0&1\\
0&0&0&1&0\end{array}\right],\mbox{ and }
\left[\begin{array}{ccccc}
0&0&0&1&0\\
0&0&0&0&1\\
0&0&1&0&0\\
1&0&0&0&0\\
0&1&0&0&0\end{array}\right]\]
are pairwise orthogonal ASM-mates.
\endproof
}
\end{example}

This extension of the notion of orthogonality leads to several questions

\begin{problem}{\rm \label{prob:orthogonalASM} 

\begin{itemize}
\item[\rm (a)] Which $n\times n$ ASMs have an orthogonal ASM-mate? 
Which $n\times n$ ASMs have a permutation matrix as an orthogonal mate?
\item[\rm (b)] 
Does there exist a pair of orthogonal $6\times 6$ ASHM-LS's? 
\item[\rm (c)] Does there exist a set of more than $(n-1)$ pairwise orthogonal  $n\times n\times n$ ASHM-LS's?
\end{itemize}

If the answer to (b) is no, it may not  be straightforward to resolve as it contains the known, but not easily verifiable,  fact that there does not exist a pair of $6\times 6$ orthogonal latin squares. One possible way to   consider the question (c) is to start with a set of 
$(n-1)$ pairwise orthogonal $n\times n$ latin squares (so coming from $(n-1)$ $n\times n\times n$ permutation ASHMs, and construct an $n\times n\times n$
ASHM-LS orthogonal to each of them.  This possibility addresses the question as to whether or not there exists a set of $n$ pairwise orthogonal ASHM-LS's extending a set of $(n-1)$ ordinary latin squares. \endproof
}
\end{problem}

If $L$ is an $n\times n$ latin square, then $L$ corresponds to a  unique $n$-tuple $(P_1,P_2,\ldots,P_n)$ of $n\times n$ permutation matrices  with $J_n=P_1+P_2+\cdots+P_n$, that is, to a unique $n\times n\times n$ permutation hypermatrix $[P_1,P_2,\ldots,P_n]$, such that $L=1P_1+2P_2+\cdots+nP_n$.  This leads to the following question for ASHM-LS's.

\begin{problem} \label{[rob:latinunique}{\rm
Let ${\mathcal A}_n$ be the set of $n\times n\times n$ ASHMs and let ${\mathcal L}_n$ be the set of $n\times n$ ASHM-LS's.
If $L\in {\mathcal L}_n$. then there exists an $n\times n\times n$ 
ASHM $A=[A_1,A_2,\ldots,A_n]\in {\mathcal A}_n$ such that $L=1A_1+2A_2+\cdots+nA_n$. Is the ASHM $[A_1,A_2,\ldots,A_n]$ unique? Equivalently, is the mapping 
\[\ell: {\mathcal A}_n\rightarrow {\mathcal L}_n\]
 given by ${\ell} (A)=L$  injective? In words,
can two different $n\times n\times n$ ASHMs give identical ASHM-LS's? Note that if ${\ell}(A)$ is an ordinary LS, then we know from Theorem \ref{th:LS} that the ASHM $A$ is a permutation hypermatrix and thus is uniquely determined.
\endproof
}\end{problem}

\end{document}